%% file: siam_arxiv.tex
\documentclass[final,onefignum,onetabnum]{siamart190516_arxiv}

\input{shared_arxiv}

\ifpdf
\hypersetup{
  pdftitle={Stable rank-adaptive Dynamically Orthogonal Runge-Kutta schemes},
  pdfauthor={Aaron Charous and Pierre F.J. Lermusiaux}
}
\fi

\myexternaldocument{supp_mat}
\begin{document}

\maketitle

% REQUIRED
\begin{abstract}
We develop two new sets of stable, rank-adaptive Dynamically Orthogonal Runge-Kutta (DORK) schemes that capture the high-order curvature of the nonlinear low-rank manifold. The DORK schemes asymptotically approximate the truncated singular value decomposition at a greatly reduced cost while preserving mode continuity using newly derived retractions.
%We develop two new sets of retractions and integrators for the dynamical low-rank approximation. These retractions capture high-order curvature of the nonlinear low-rank manifold by asymptotically approximating the truncated singular value decomposition at a greatly reduced cost while preserving mode continuity. 
%First, optimal perturbative retractions are derived. 
We show that arbitrarily high-order optimal perturbative retractions can be obtained, and we prove that these new retractions are stable. In addition, we demonstrate that repeatedly applying retractions yields a gradient-descent algorithm on the low-rank manifold that converges superlinearly when approximating a low-rank matrix.
When approximating a higher-rank matrix, iterations converge linearly to the best low-rank approximation. We then develop a rank-adaptive retraction that is robust to overapproximation. Building off of these retractions, we derive two rank-adaptive integration schemes that dynamically update the subspace upon which the system dynamics are projected within each time step: the 
% textbf{s}table \textbf{a}daptive-rank \textbf{d}ynamically \textbf{o}rthogonal \textbf{R}unge-\textbf{K}utta 
stable, optimal Dynamically Orthogonal Runge-Kutta
(so-DORK) and 
%gradient-\textbf{descent} \textbf{d}ynamically \textbf{o}rthogonal \textbf{R}unge-\textbf{K}utta  
gradient-descent Dynamically Orthogonal Runge-Kutta
(gd-DORK) schemes. These integration schemes are numerically evaluated and compared on an ill-conditioned matrix differential equation, an advection-diffusion partial differential equation, and a nonlinear, stochastic reaction-diffusion partial differential equation. Results show a reduced error accumulation rate with the new stable, optimal and gradient-descent integrators. In addition, we find that rank adaptation allows for highly accurate solutions while preserving computational efficiency.

\end{abstract}

% REQUIRED
\begin{keywords}
  Dynamical low-rank approximation, DO equations, matrix differential equations, stochastic differential equations, uncertainty quantification, reduced-order modeling, differential geometry, retraction, fixed-rank matrix manifold
\end{keywords}

% REQUIRED
\begin{AMS}
  54C15, 65F55, 53B21, 15A23, 57Z05, 57Z20, 57Z25, 60G60, 65C30, 65M12, 65M22, 65C20, 81S22, 94A08, 53A07, 35R60
  
  %Retraction
  %Numerical methods for low-rank matrix approximation; matrix compression
  %Methods of local Riemannian geometry
  %Factorization of matrices
  %Relations of manifolds and cell complexes with physics
  %Relations of manifolds and cell complexes with engineering
  %Relations of manifolds and cell complexes with computer and data science
  %Random fields
  %Numerical solutions to stochastic differential and integral equations
  %Stability and convergence of numerical methods for initial value and initial-boundary value problems involving PDEs
  %Numerical solution of discretized equations for initial value and initial-boundary value problems involving PDEs
  %Probabilistic models, generic numerical methods in probability and statistics
  %Open systems, reduced dynamics, master equations, decoherence
  %Image processing
  %Higher-dimensional and -codimensional surfaces in Euclidean and related $n$-spaces
  %PDEs with randomness, stochastic partial differential equations
\end{AMS}

\section{Introduction}
\label{sec:intro}

Scientific computing often involves numerical simulations which overwhelm available computational resources. A common reduced-order modeling technique is to use low-rank approximations of matrices and/or tensors to improve computational efficiency and reduce storage requirements. Mathematically, we seek $X(t)$, a low-rank approximation to our full state $\Xf(t)$, such that at all discrete times $t_i$, $X(t_i)$ is the best approximation to $\Xf(t_i)$. That is, 
\begin{gather*}
    X(t_i) = \argmin_{\tilde{X}(t_i)\in \Mr} ||\tilde{X}(t_i)-\Xf(t_i)|| \quad \forall t_i,
\end{gather*}
where $\Mr$ denotes the manifold of rank-$r$ matrices (see Table\;\ref{tab:notation} for key notation).

The dynamical low-rank approximation (DLRA) \cite{dlra} is a method to instantaneously optimally evolve a system's low-rank approximation for common time-dependent partial differential equations (PDEs) \cite{feppon_lermusiaux_SIMAX2018a}. Suppose we are given a discretized dynamical system, stochastic or deterministic, as a first-order differential equation
\begin{gather}
    \frac{d \Xf}{dt} = \LL(\Xf,t;\omega) \label{eq:fullRankSystem},
\end{gather}
where $\omega \in \Omega$ denotes a simple event in a stochastic event space. 
Note that any order differential equation may be rewritten as a first-order system, so this is not at all restrictive (see, e.g., \cite{hochbruckdynamical}). There are two common approaches to solve for $X$. The first is to derive an instantaneously optimal differential equation for $X$ in terms of $\LL$ given that $X$ is restricted to the low-rank manifold. Using the Dirac-Frenkel time-dependent variational principle (see Table \ref{tab:keyExpressions} for the incurred approximation error), we arrive at the following differential equation for our low-rank approximation.
\begin{gather}
    \frac{d X}{dt} = \PT [\LL(X,t;\omega)] \label{eq:tanSpaceDiffEq}
\end{gather}
For a derivation, we refer to \cite{dlra,feppon_lermusiaux_SIMAX2018a,charous_MSThesis2021,charous_lermusiaux_SJSC2023a}. Next, a parameterization of the low-rank matrix $X$ is chosen; two common parameterizations are $X = UZ^T$ with orthonormal modes $U\in \Stmr$ and coefficients $Z \in \Rr$ as in \cite{feppon_lermusiaux_SIMAX2018a,feppon_lermusiaux_SIREV2018,charous_lermusiaux_Oceans2021,charous_lermusiaux_SJSC2023a}, or $X = USV^T$ similar to a classic singular value decomposition (SVD) with $U \in \Stmr$, $S \in \Rrr$, and $V \in \Stnr$ as in \cite{dlra,projSplit,lermusiaux_robinson_MWR1999,lermusiaux_MWR1999,lermusiaux_JMS2001}.
With the parameterization, evolution equations for $U$ and $Z$ (or $U$, $S$, and $V$) are derived with a choice of gauge in order to eliminate redundant degrees of freedom and simplify the equations. 
For instance, the dynamically orthogonal (DO) condition insists that the rate-of-change $\Dot{U} \in \DOU$ (see table \ref{tab:notation}) which, from (\ref{eq:tanSpaceDiffEq}), yields the following equations.
\begin{gather}
\begin{aligned}
    &\Dot{U} = \PUp \LL(UZ^T,t;\omega) Z (Z^T Z)^{-1}\\
    &\Dot{Z} = \LL(UZ^T,t;\omega)^T U
\end{aligned} \label{eq:DOeqns}
\end{gather}
Alternatively, insisting that $\Dot{U} \in \TStmr$ and $\Dot{V} \in \TStnr$ yields the following differential equations for $U,S,$ and $V$.
\begin{gather}
\begin{aligned}
    &\Dot{U} = \PUp \LL(USV^T,t;\omega) V S^{-1}\\
    &\Dot{S} = U^T \LL(USV^T,t;\omega) V\\
    &\Dot{V} = \PVp \LL(USV^T,t;\omega)^T U S^{-T}
\end{aligned} \label{eq:DLRAeqns}
\end{gather}
Finally, different numerical schemes may be employed to solve (\ref{eq:DOeqns}) or (\ref{eq:DLRAeqns}). Numerical integration schemes include the projector-splitting integrator \cite{projSplit}, its closely related variant \cite{kslVariant}, and similar approaches based in local coordinates \cite{localCoord1,localCoord2}.

A second approach is to employ so-called projection methods (see, e.g., \cite{projectionMethods1,projectionMethods2}). Instead of numerically integrating $\PT \LL$, the full $\LL$ is integrated and then projected back onto the low-rank. 
% We may derive the numerical scheme by first discretizing time in
Integrating (\ref{eq:fullRankSystem}) in time, we obtain
\begin{gather}
    \Xf_{i+1} = \Xf_{i} + \Delta t \LLo . \label{eq:intScheme}
\end{gather}
Here, $\LLo = \frac{1}{\Delta t}\int_{t_i}^{t_{i+1}} \LL(\Xeither(t),t;\omega) dt + \BigO(\Delta t ^k)$ is either the exactly integrated differential operator or the approximately integrated differential operator up to order $k$.
In the exact case, the integrand $\LL$ in $\LLo$ depends on the unknown future solution $\Xeither$;
furthermore, $\Xeither$ may be the full-rank system $\Xf$ or the low-rank system $X$, though the latter is another approximation inducing a dynamical model closure error \cite{charous_lermusiaux_SJSC2023a}, $\ed \equiv \LL(\Xf) - \LL(X)$. In defining $\LLo$ this way, (\ref{eq:intScheme}) may represent any chosen integration scheme. 
For implicit schemes, a system would need to be solved \cite{charous_PhDThesis2023,charous_lermusiaux_SJSC2023c} and only explicit schemes are considered here.
%\AC{2023c is not defined as we have not submitted it. Do we include it?} 
%\PFJL{I would since it may be submitted by the time this siam paper appears? You can check if I got it right, and I also added your thesis.}
%
Letting $X_{i+1}^* \in \Mr$ denote the best possible approximation of $\Xf_{i+1}$ given information at time $t_i$, we have
\begin{gather*}
    X_{i+1}^* = \PM(\Xf_i + \Delta t \LLo).
\end{gather*}
Of course, we may not have access to $\Xf_i$; we substitute in $X_i \in \Mr$ instead as an approximation.
\begin{gather}
    X_{i+1}^* = \PM(X_i + \Delta t \LLo) \label{eq:bestProjMethod}
\end{gather}
Unfortunately, $\PM$, i.e.\ the truncated SVD, is expensive to compute every time step \cite[p. 237]{trefethen1997numerical}. As such, we seek approximations to the exact projection operator. With the use of (extended) retractions, which map matrices from the embedding Euclidean space back to the manifold, see \cite{retSurvey,charous_lermusiaux_SJSC2023a}, the numerical scheme then becomes
\begin{gather}
    X_{i+1} = \Ret_{X_i}(\Delta t \LLo) \label{eq:projectionMethod},
\end{gather}
where $X_{i+1}$ is no longer the \emph{best} possible approximation given information at time $t_i$. The difference between (\ref{eq:bestProjMethod}) and (\ref{eq:projectionMethod}) defines the \emph{projection-retraction error} \cite{charous_lermusiaux_SJSC2023a},
\begin{gather}
    \epr \equiv \Ret_{X_i}(\Delta t \LLo) - \PM(X_i + \Delta t \LLo). \label{eq:projRetError}
\end{gather}

Using both an integration scheme for $\LLo$ and a retraction of order $k$, whereby $\epr = \BigO(\Delta t^{k+1})$, guarantees $k$th-order convergence to the best low-rank approximation assuming that the difference between $X\in\Mr$ and the full-rank system $\Xf$ is small. Integration schemes that use projection methods include those presented in \cite{charous_lermusiaux_SJSC2023a} as well as projected Runge-Kutta methods \cite{rkRets}, which actually use both formulations of the problem. The use of retractions and projection methods is general and may be applied to any integration scheme for $\LLo$. 

To wrap up, we note the difference in the derivation of \say{tangent-space-integration methods,} i.e.\ schemes that integrate (\ref{eq:tanSpaceDiffEq}) (e.g.\ \cite{projSplit,kslVariant,localCoord1,localCoord2}), and projection methods of the form (\ref{eq:projectionMethod}) (e.g.\ \cite{projectionMethods1,projectionMethods2}). To derive (\ref{eq:tanSpaceDiffEq}), the Dirac-Frenkel time-dependent variational principle \cite{dirac,lubichbook} is applied in the continuous time setting, and then time is discretized for numerical integration. 
For projection methods, there are two approaches, both of which start by discretizing time for $\Xeither$ before making any approximations.
The first approach is to integrate the full-rank differential operator $\LL(\Xf(t),t)$ in the embedding Euclidean space. Then, $\Xf$ is approximated as $X$ as in (\ref{eq:bestProjMethod}), and the projection operator is approximated with a retraction as in (\ref{eq:projectionMethod}). 
The second approach is to apply an integration scheme to $\LL(X(t),t)$; here we integrate along the manifold, so retractions may be applied as we compute $\LLo$. At the last stage, a retraction is applied as in (\ref{eq:projectionMethod}). So, the difference between the tangent-space-integration and projection methods comes from the order of time discretization for $\Xeither$ and approximation, which manifests in whether or not the tangent-space projection operator $\PT$ is explicitly integrated. Nevertheless, both schemes form consistent integrators.
%\vskip -0.5truecm
\begin{table}
% \caption{Please write your table caption here}
\begin{threeparttable}

\caption*{\small \textbf{Notation}}
\small
\begin{tabular}{ll}
$\omega$ & Simple event\\
$\Omega$ & Event space\\
$\Xf$ & Full-rank system state in $\mathds{R}^{m\times n}$\\
$X$ & Low-rank approximation of system state in $\Mr$\\
$\Xeither$ & Either the full-rank $\Xf$ or low-rank $X$ system state\\
$\newPoint$ & Next system state point before retraction, $U_iZ_i^T + \Delta t \LLo$ \\
$\Uo$ & Time-averaged change in $U$ integrated from $t_{i}$ to $t_{i+1}$ such that $U_{i+1} = U_i + \Delta t \Uo$\\
$\Zo$ & Time-averaged change in $Z$ integrated from $t_{i}$ to $t_{i+1}$ such that $Z_{i+1} = Z_i + \Delta t \Zo$\\
$\|\bullet \|$ & Frobenius norm\\
$\Mr$ & Manifold of rank-r $m\times n$ real matrices\\
$\LLo$ & Approximately or exactly integrated differential operator defined as \\ &  $\frac{1}{\Delta t}\int_{t_i}^{t_{i+1}} \LL(\Xeither(t),t) dt + \BigO(\Delta t ^k)$\\
$\PM$ & Projection onto $\Mr$\\
$\PMp$ & Orthogonal complement of projection onto $\Mr$ defined by $I - \PM$\\
$\TX$ & Tangent space of low-rank manifold at point $X$\\
$\PT$ & Projection operator onto tangent space of low-rank manifold at point $X$\\
$\Stmr$ & Stiefel manifold of $m\times r$ matrices\\
$\Stnr$ & Stiefel manifold of $n\times r$ matrices\\
$\Rr$ & Set of rank-r matrices in $\mathds{R}^{n\times r}$\\
$\Rrr$ & Set of rank-r matrices in $\mathds{R}^{r\times r}$\\
$\DOU$ & DO space (which depends on $U$) defined as $\left\{\du \in \mathds{R}^{m\times r}: U^T\du = 0\right\}$\\
$so(r)$ & Space of skew-symmetric, real $r\times r$ matrices\\
$\TStmr$ & Tangent space of $m\times r$ Stiefel manifold equivalent to $\left\{\du \in \mathds{R}^{m \times r}: U^T\du \in so(r)\right\}$\\
$\TStnr$ & Tangent space of $n\times r$ Stiefel manifold equivalent to $\left\{\dv \in \mathds{R}^{m \times r}: V^T\dv \in so(r)\right\}$\\
$\PU$ & Projection onto modes $U$ defined by $UU^T$\\
$\PV$ & Projection onto modes $V$ defined by $VV^T$\\
$\PUp$ & Orthogonal complement of projection onto modes of $U$ defined by $(I-UU^T)$\\
$\PVp$ & Orthogonal complement of projection onto modes of $V$ defined by $(I-VV^T)$\\ 
$\Ret_{A}(\Delta t B)$ & Retraction from point $A$ in direction $B$ such that $\Ret_{A}(\Delta tB) = A + \Delta t \PTA B + \BigO(\Delta t^2)$\\
$\Ret^{(p)}$ & Retraction that asymptotically approximates $\PM$ to the $p^{\rm th}$ order. 
\end{tabular}
\end{threeparttable}
\caption{\small 
%In this table, we define 
Relevant mathematical structures and operators referenced within the text.
}
\label{tab:notation}       % Give a unique label
\end{table}

%As we have noted, 
%Significant 
Progress has been made recently in developing tangent-space-integration methods \cite{projSplit,kslVariant,localCoord1,localCoord2,hochbruckdynamical}. Arbitrarily high-order integration schemes may be derived via symmetric composition with the adjoint method. However, this integrator may lack desirable properties of other schemes for particular physical systems. 
For instance, one may seek a symplectic integrator \cite{musharbash2020symplectic} for Hamiltonian systems, a total-variation diminishing scheme \cite{gottlieb1998total} to avoid oscillations in, e.g., shock propagation problems, schemes that conserve relevant physical quantities, or implicit-explicit schemes to preserve stability. This has been studied in model-order reduction \cite{celledoni2002class,zimmermann2018geometric} and recently in the context of the DLRA \cite{pagliantini2021dynamical,hesthaven2022rank}. Projection methods allow easy construction of new low-rank integration schemes with our desired features simply by implementing pre-existing full-rank schemes with the addendum of applying a high-order retraction as a final step to project back to the low-rank manifold. One may also use projections/retractions after any given step of a full-rank integration scheme in order to limit the quickly growing rank of $\LLo$ as is done with projected Runge-Kutta methods in \cite{rkRets}. Without a particular choice of gauge, all of these schemes lack mode continuity \cite{ueckermann_et_al_JCP2013,feppon_lermusiaux_SIREV2018,lin_lermusiaux_NM2021}, i.e.\ the vectors in $U$ should only change slowly as a function of time so that their dynamics are interpretable. Such a feature is essential for the reduced-order evolution of dynamical systems in time, especially for uncertainty quantification. As such, we require efficient and accurate retractions which asymptotically approximate the projection operator and preserve mode continuity as a tool to construct flexible integration schemes.

In \cite{charous_lermusiaux_SJSC2023a}, we developed a new set of perturbative retractions that approximate the truncated SVD. From these retractions, we created Dynamically Orthogonal Runge-Kutta (DORK) integration schemes, which updated the subspace upon which we projected the system dynamics as we integrated, yielding highly accurate schemes. Nevertheless, integration using these perturbative retractions can have three issues: (i) they can lead to overshooting if time steps are too large, (ii) their higher-order corrections become more expensive as greater accuracy is required, and (iii) they can break down in the presence of small singular values. The tangent-space-integration methods mentioned are robust to small singular values, which enable the recent development of rank-adaptive methods. In \cite{ceruti2022rank}, the rank of the solution is doubled during the integration of each time step, the solution is integrated, and then the rank is reduced. While this algorithm is effective, it has the computational cost of always computing a solution at double the rank, which is undesirable. There are a few potential algorithmic remedies in the literature to reduce the computational cost by determining when the rank should be increased. Rank-adaptive algorithms proposed in \cite{gao2022riemannian,lin_PhDThesis2020} recommend calculating the norm of the residual, namely $\|(I-\PT)\LLo\|$, at each time step; if this value is large, the rank should be augmented. Similarly, \cite{zhou2016riemannian} recommend using the tangent of the angle between $\PT \LLo$ and $\LLo$ as a metric. Rank-adaptive algorithms are also developed in the context of dynamically orthogonal differential equations in \cite{sapsis_lermusiaux_PHYSD2012,feppon_lermusiaux_SIREV2018}. However, the perturbative retractions and DORK schemes in \cite{charous_lermusiaux_SJSC2023a} are incompatible with rank adaptation because the rank augmentation necessitates the integration of small singular values which will lead to the inversion of an ill-conditioned correlation matrix $Z^TZ$.

\begin{table}
% \caption{Please write your table caption here}
\begin{threeparttable}

\caption*{\small \textbf{Key expressions}}
\small
\begin{tabular}{ll}
$\{\Ud_j^{\text{orig}}, \Zd_j^{\text{orig}}\}$ & Original first- through fourth-order corrections, see (5.4-5.7) in \cite{charous_lermusiaux_SJSC2023a}\\
$\{\Ud_j\}_{j=1}^4$ & Optimal first- through fourth-order corrections (\ref{eq:corr1}-\ref{eq:corr4})\\
$U_{i+1}^{\text{robust}}$ & Optimal, robust first-order update to $U_i$ (\ref{eq:uUpdateRobust}) \\
$Z_{i+1}^{\text{optimal}}$ & Optimal update to $Z_i$ (\ref{eq:zUpdate})\\
$\{\Ud_j, \Ld_j\}_{i=1}^4$ & First- through fourth-order so-DORK schemes (\ref{eq:corr1_so-DORK}-\ref{eq:corr4_so-DORK})\\
$\epr$ & Projection-retraction error defined as $\Ret_{X}(\Delta t \LLo) - \PM(X + \Delta t \LLo)$\\
$\el$ & Local retraction error defined as $\Ret_{X}(\Delta t \LLo) - (X + \Delta t \LLo)$\\
$\en$ & Normal closure error defined as 
$\PMp(\Xf)= 
\Xf - \PM(\Xf)$\\
$\ed$ & Dynamical model closure error defined as $\LL(\Xf) - \LL(X)$\\
$\edf$ & Dirac-Frenkel model closure error defined as $\LL(\Xf) - \PT \LL(X)$\\
$\et$ & Total error due to dynamical low-rank approximation defined as $\|\Xf - X\|$
\end{tabular}
\end{threeparttable}
\captionsetup{font=footnotesize,labelfont=footnotesize}
\caption{
%Here, we summarize the 
Novel retractions derived and error metrics used throughout this paper. These retractions and error metrics arise due to the dynamic curvature of the low-rank manifold.
}
\label{tab:keyExpressions}       % Give a unique label
\end{table}
%\vskip -0.5truecm

In this paper, we remedy these three issues with a novel set of optimal perturbative retractions. In section \ref{sec:optPert}, we show these retractions are more accurate than those previously introduced at a reduced computational cost and, importantly, largely avoid overshoot. In section \ref{sec:gradDescent} we introduce a new simple yet effective set of retractions composed of perturbative retractions, which we call \say{gradient-descent retractions.} Gradient-descent retractions yield exceptional accuracy up to arbitrary order without a quickly growing computational cost. We also develop a simple \say{automatic} gradient-descent retraction that determines how many iterations to compute based on a user-given tolerance. 
In section \ref{sec:rankAdaptive}, we develop perturbative retractions robust to small singular values that may be coupled with our gradient-descent retractions, and we include rank adaptation as a feature. 
Then, in section \ref{sec:newDORK} we derive two new low-rank integration schemes based off our new retractions, 
the stable, optimal Dynamically Orthogonal Runge-Kutta (so-DORK) schemes and 
gradient-descent Dynamically Orthogonal Runge-Kutta
(gd-DORK) schemes. 
%the \textbf{a}daptive-rank \textbf{d}ynamically \textbf{o}rthogonal \textbf{R}unge-\textbf{K}utta st\textbf{able} (ADORKABLE) schemes and the \textbf{d}ynamically \textbf{o}rthogonal \textbf{R}unge-\textbf{K}utta gradient-\textbf{descent} (DORK descent) schemes.
We demonstrate the efficacy of our retractions and integrators with numerical experiments on a matrix differential equation, a linear partial differential equation, and a nonlinear, stochastic partial differential equation in section \ref{sec:numExp}. Finally, we conclude in section \ref{sec:conc} with some remarks on the utility of these new retractions and integrators, and we discuss future research directions. Throughout the manuscript, we use the terms \say{high-order retraction} or $p$th-order retraction to denote retractions that have projection-retraction error (see table \ref{tab:keyExpressions}) of order $\BigO(\Delta t^{p+1})$, which is in contrast to a \say{classical} second-order retraction, whose second-order error belongs to the normal space of $\Mr$.

%----------------------------------------------
%------------
\section{Optimal perturbative retractions}
\label{sec:optPert}
\subsection{Derivation}
In this section, we will derive \say{optimal} perturbative retractions for explicit integration schemes, an improvement over the perturbative retractions developed in \cite{charous_lermusiaux_SJSC2023a}. In the original set of perturbative retractions, we approximately solved the following optimization problem.
\begin{equation}
\begin{aligned}
    &\argmin_{U_{i+1}, Z_{i+1}} \left\|U_iZ_i^T + \Delta t \LLo - U_{i+1}Z_{i+1}^T\right\| \label{eq:opt1}\\
    &\phantom{a}\text{st}\phantom{a}U_{i+1} = U_{i} + \Delta t \Uo, \quad \Uo \in \DOU\\
    &\phantom{asta} Z_{i+1} = Z_i + \Delta t \Zo
\end{aligned}
\end{equation}
Solving for the first-order optimality conditions of the system yields nonlinear matrix equations, but writing $\Uo$ and $\Zo$ as perturbation series in $\Delta t$ yields linear equations to solve, giving our retractions. For notational convenience, we introduce $\newPoint$, 
\begin{equation*}
\newPoint \equiv U_iZ_i^T + \Delta t \LLo\,,
\end{equation*}
which denotes the next point we wish to approximate before we retract back to the low-rank manifold (see table \ref{tab:notation}).
Next, we modify the optimization problem (\ref{eq:opt1}) by solving for $Z_{i+1}$ directly and enforcing numerical orthonormality. 

First, suppose that $U_{i+1} \in \Stmr$ is given a priori. In a manner similar to an alternating least squares procedure \cite{ALS}, we can exactly solve for $Z_{i+1}$ in the following unconstrained optimization problem.
\begin{equation}
    \argmin_{Z_{i+1}} \left\|\newPoint - U_{i+1}Z_{i+1}^T\right\| \quad \Rightarrow \quad Z_{i+1} = \newPoint^TU_{i+1} \label{eq:zeq}
\end{equation}

Second, we now would like to substitute this expression for $Z_{i+1}$ in as a constraint in (\ref{eq:opt1}) to yield a more accurate set of retractions. But, we must first recognize that despite the DO condition, $U_i + \Delta t \Uo$ is not exactly orthonormal: $(U_i + \Delta t \Uo)^T(U_i + \Delta t \Uo) = I + \BigO(\Delta t^2)$. So we cannot simply write $U_{i+1} = U_i + \Delta t \Uo$. A re-orthonormalization procedure is typically applied; see \cite{lin_lermusiaux_NM2021} for a fast algorithm. This re-orthonormalization amounts to right-multiplying $U_i + \Delta t \Uo$ by some nonsingular $r\times r$ matrix, $\Reo$, such that $\Reo^T(U_i + \Delta t \Uo)^T (U_i + \Delta t \Uo)\Reo = I$. Then, we may let $U_{i+1} = (U_i + \Delta t \Uo)\Reo$, and we can rewrite our optimization problem as follows.

\begin{equation}
\begin{aligned}
    &\argmin_{U_{i+1}} \left\|\newPoint - U_{i+1}Z_{i+1}^T\right\| \label{eq:opt2}\\
    &\hspace{0.75ex}\text{st}\phantom{a}U_{i+1} = (U_{i} + \Delta t \Uo)\Reo, \quad \Uo \in \DOU\\
    &\hspace{0.75ex}\phantom{sta} Z_{i+1} = \newPoint^TU_{i+1},\quad \Reo\in\Rrr
\end{aligned}
\end{equation}
With this formulation, we are now only optimizing over one matrix, $U_{i+1}$ (or $\Uo$, equivalently). Additionally, we are ensuring that whatever matrix $U_{i+1}$ we end up with, $Z_{i+1}$ in (\ref{eq:zeq}) gives the \emph{optimal} low-rank approximation to $\newPoint$. We remark that the update equation for the coefficients is equivalent to that used in proper orthogonal decomposition \cite{pod1,pod2} and other reduced-basis methods such as dynamic mode decomposition \cite{dmd1,dmd2,dmd3,dmd4}.

By substituting our constraints into the cost function in (\ref{eq:opt2}) and differentiating with respect to $U_{i+1} = (U_i + \Delta t \Uo)G$, the first-order optimality condition can be simplified to the following nonlinear matrix equation we seek to solve.
\begin{gather*}
    \frac{\partial}{\partial [(U_i + \Delta t \Uo)G]} \left[\left\| \newPoint - (U_i + \Delta t \Uo)GG^T(U_i + \Delta t \Uo)^T\newPoint\right\|^2\right] = 0 \\
    \Leftrightarrow \left[I-(U_i + \Delta t \Uo)\Reo\Reo^T(U_i + \Delta t \Uo)^T\right]\newPoint\newPoint^T(U_i + \Delta t \Uo) = 0
\end{gather*}
We may remove the $\Reo\Reo^T$ dependence by using the definition of $G$ to obtain
\begin{align*}
    \Reo\Reo^T &= \left[(U_i + \Delta t \Uo)^T(U_i + \Delta t \Uo)\right]^{-1}\\
    &= \left[I + \Delta t^2 \Uo^T\Uo\right]^{-1}\\
    &= I - \Delta t^2 \Uo^T\Uo + \Delta t^4 (\Uo^T\Uo)^2 - \cdots .
\end{align*}
The last line follows from the Neumann series of the inverse of $I + A$ for some matrix $A$, which converges assuming that $\Delta t^2 \|\Uo^T\Uo\| < 1$; this is not a restrictive assumption since $\Delta t$ can always be chosen to be smaller such that this holds, and this is further remedied by our robust retractions in section \ref{sec:rankAdaptive}. Note that since $\Reo$ has been eliminated from our equation, it need not be explicitly computed but will be computed implicitly when re-orthonormalizing $U_{i+1}$. Then our optimality condition becomes the following.
\begin{gather}
    \left[I-(U_i + \Delta t \Uo)(I - \Delta t^2 \Uo^T\Uo + \cdots)(U_i + \Delta t \Uo)^T\right]\newPoint\newPoint^T(U_i + \Delta t \Uo) = 0 \label{eq:optCondition}
\end{gather}
Now, we assume that $\Uo$ may be written as a perturbation series in $\Delta t$,
\begin{gather}
    \Delta t \Uo = \sum_{j=1}^\infty \Delta t^{j} \Ud_{j} . \label{eq:pertSeries}
\end{gather}
After substituting (\ref{eq:pertSeries}) into (\ref{eq:optCondition}), we may solve for $\Ud_j$ via small linear systems by grouping terms of the same order of $\Delta t$. We refer to \cite{charous_MSThesis2021, charous_lermusiaux_SJSC2023a} for an analogous derivation. Below, we provide first- through fourth-order corrections.
\begingroup
\allowdisplaybreaks
\begin{subequations}
\begin{align}
    &\Ud_1 = \PUnp \LLo Z_i (Z_i^TZ_i)^{-1} \label{eq:corr1}\\
    &\Ud_2 = \left[\PUnp\LLo\LLo^TU_i - \Ud_1\left(U_i^T \LLo Z_i +  Z_i^T \LLo^T U_i\right) \right](Z_i^TZ_i)^{-1} \label{eq:corr2}\\
    &\begin{aligned}
    &\Ud_3 = \left[\PUnp \LLo \LLo^T \Ud_1 -
    \Ud_2 \left(U_i^T \LLo Z_i +  Z_i^T \LLo^T U_i\right) \right.\\&\phantom{\Ud_4}-
    \Ud_1\left. \left(U_i^T\LLo\LLo^T U_i + \Ud_1^T\LLo Z_i + Z_i^T\LLo^T\Ud_1 - \Ud_1^T\Ud_1 Z_i^T Z_i  \right)\right](Z_i^T Z_i)^{-1} \label{eq:corr3}
    \end{aligned}\\
    &\Ud_4 = \left[\PUnp\LLo\LLo^T\Ud_2 - \Ud_3\left(U_i^T \LLo Z_i +  Z_i^T \LLo^T U_i\right) \right. \notag \\ &\phantom{\Ud_4}\begin{aligned}
        &-\Ud_2\left(U_i^T\LLo\LLo^T U_i + \Ud_1^T\LLo Z_i + Z_i^T\LLo^T\Ud_1 - \Ud_1^T\Ud_1 Z_i^T Z_i  \right)  \\
        &-\Ud_1\left(U_i^T\LLo\LLo^T\Ud_1 + \Ud_1\LLo\LLo^T U_i + \Ud_2^T\LLo Z_i + Z_i^T\LLo^T\Ud_2 - \Ud_1^T\Ud_2 Z_i^TZ_i  \right. \end{aligned} \label{eq:corr4}\\ 
    &\phantom{\Ud_4} - \left. \left. \Ud_2^T\Ud_1 Z_i^T Z_i - \Ud_1^T\Ud_1 U_i^T\LLo Z_i - \Ud_1^T\Ud_1Z_i^T\LLo^TU_i \right)\right](Z_i^T Z_i)^{-1} \notag  %\left. \left. is in a weird place so that equations align correctly
\end{align}
\label{eq:retract_cor}
\end{subequations}
\endgroup
We note here that there are two patterns to exploit for computational efficiency in (\ref{eq:corr1}-\ref{eq:corr4}). First, the high-order corrections require the lower-order corrections, and each term on the right-hand side that is left-multiplied by $\Ud_j$ is repeated in the higher-order corrections. So, the repeated terms should be stored to avoid extra matrix multiplications. Second, several of the terms above are transposes of each other, so again, storage will avoid redundant computation.

For concreteness, we explicitly write out a second-order retraction with our optimal perturbative retraction.
\begin{equation}
\begin{gathered}
    X_{i+1} = \Ret^{\text{opt-2}}_{X_i}(\Delta t \LLo) = U_{i+1}Z_{i+1}^T\\
    U_{i+1} = \texttt{orth}\left(U_i + \Delta t \Ud_1 + \Delta t^2 \Ud_2 \right),
    \qquad Z_{i+1} = (X_i + \Delta t \LLo)^T U_{i+1}
\end{gathered} \label{eq:second_order_optimal}
\end{equation}
The first- and second-order corrections, $\Ud_1$ and $\Ud_2$, respectively, are given by equations (\ref{eq:corr1}) and (\ref{eq:corr2}), and the orthonormalization procedure --- which implicitly generates the $G$ matrix referenced beforehand --- may be a QR decomposition or other fast algorithm as in \cite{lin_lermusiaux_NM2021}.

It is instructive to derive the \emph{local retraction error}, $\el$, of (i) any generic retraction and (ii) our optimal pertubative retractions in (\ref{eq:corr1}-\ref{eq:corr4}). From the generic definition of a retraction in Table \ref{tab:notation}, any local retraction error may be written as
\begin{align*}
    \el &\equiv  X_i + \Delta t \LLo - \Ret_{X_i}(\Delta t \LLo)\\
    &= U_iZ_i^T + \Delta t \LLo - \left(U_iZ_i^T + \Delta t \PTn\LLo + \BigO(\Delta t^2)\right)\\
    &= \left(I - \PTn\right)\Delta t \LLo + \BigO(\Delta t^2).
\end{align*}
Explicitly substituting in our expression for $\PTn$ (see \cite{feppon_lermusiaux_SIMAX2018a}), we obtain
\begin{align}
    \el &= (I-U_iU_i^T)\Delta t \LLo (I - Z_i(Z_i^T Z_i)^{-1}Z_i^T) + \BigO(\Delta t^2) \label{eq:elGeneral}
\end{align}

For our optimal perturbative retractions, we may write the local retraction error as
\begin{align*}
    \el^{\text{opt}} = U_iZ_i^T + \Delta t \LLo - U_{i+1}Z_{i+1}^T,
\end{align*}
noting that $U_{i+1}Z_{i+1}^T = \Ret^{\text{opt}}_{X_i}(\Delta t \LLo)$ by definition. Now, we substitute in our expression for $Z_{i+1}$ from equation (\ref{eq:zeq}).
\begin{align}
    \el^{\text{opt}} &= U_iZ_i^T + \Delta t \LLo - U_{i+1}U_{i+1}^T(U_iZ_i^T + \Delta t \LLo) \nonumber\\
    &= (I - U_{i+1}U_{i+1}^T)\newPoint \label{eq:elOptimal}
\end{align}
Contrasting the generic retraction error (\ref{eq:elGeneral}) with the optimal perturbative local retraction error (\ref{eq:elOptimal}), we see that our particular choice of $Z_{i+1}$ has allowed us to write a cleaner, more interpretable form of $\el$. By insisting that $Z_{i+1}$ is optimal given $U_{i+1}$, our error is given as the $U_{i+1}$ orthogonal projection of our new point. In words, our error is given by the information not captured by the span of the to-be-computed $U_{i+1}$. The best $U_{i+1}$ would then be the $r$ left singular vectors of $\chi$ with the largest singular values, but, of course, this would be expensive to compute. Because the truncated SVD is the exact solution to (\ref{eq:opt2}) which our retractions approximately solve, the left singular vectors are precisely what our perturbation series asymptotically approaches.

Lastly, we consider the stability of the retractions.
\begin{lemma}
Let $X_{i+1} = \Ret^{\text{opt}}_{X_i}(\Delta t \LLo)$ where we use the optimal perturbative retractions given by equations (\ref{eq:corr1}-\ref{eq:corr4}) and (\ref{eq:zeq}). Then $\|X_{i+1}\| \leq \|X_i + \Delta t \LLo\|$. \label{lem:stability}
\end{lemma}
%
%\begin{proof}
%Let $X_{i+1} = U_{i+1}Z_{i+1}^T$. By construction, $U_{i+1} \in \Stmr$, so $\|X_{i+1}\| = \|Z_{i+1}\|$. By (\ref{eq:zeq}), $\|Z_{i+1}\| = \|(X_{i} + \Delta t \LLo)^T U_{i+1}\| \leq \|X_{i} + \Delta t \LLo\|$.
%\end{proof}
%
\begin{proof}
Let $X_{i+1} = U_{i+1}Z_{i+1}^T$. 
Using the optimality (\ref{eq:zeq}), $X_{i+1} = U_{i+1}\, U_{i+1}^T (X_{i} + \Delta t \LLo)$.
By construction, $U_{i+1} \in \Stmr$, so $\|X_{i+1}\|  \leq \|X_{i} + \Delta t \LLo\|$.
\end{proof}
\qed
We remark that this stability property is unique to the \emph{optimal} perturbative retractions as it relies on the coefficients $Z_{i+1}$ being calculated as the projection of the next state onto a semi-orthonormal basis via equation (\ref{eq:zeq}). 
This is in contrast to the original perturbative retractions in which the $Z_{i+1}$'s had their own perturbation series.

%-------------------
\subsection{Illustration}
\label{sec:optimal-pert-ex}

We now exemplify the projection-retraction error of the optimal perturbative retractions (see eqs.\ (\ref{eq:pertSeries}, \ref{eq:retract_cor}) and Table \ref{tab:keyExpressions}). 
For comparison, we include the error from the original adaptive perturbative algorithm (using a hyperparameter $\varepsilon = 0.1$) from \cite{charous_lermusiaux_SJSC2023a} which automatically determines the best order of retraction to apply.
Overshoot typically occurs with the original high-order perturbative retractions
\cite{charous_lermusiaux_SJSC2023a} 
when $\Delta t$ is large because the perturbation series does not converge quickly. From another point of view, these high-order perturbations project the dynamics onto polynomial approximations of the low-rank manifold, but such polynomial approximations grow inaccurate for large steps. To avoid overshoot and ensure the perturbation series does not diverge catastrophically, the adaptive algorithm measures the norm of the high-order corrections and only includes terms that have a norm less than some user-given $\varepsilon$. Though both the original and optimal perturbative retractions orthonormalize the modes after each time step (and hence $\|U_{i+1}\| = 1$), only the optimal perturbative retractions prevent the coefficients $Z_{i+1}$ from growing too large. That is, our optimal retractions tend to avoid this overshoot (or at least limit it) due to their optimal choice of coefficients. 
This choice of coefficients is made by possible without pseudoinversion due to the re-orthonormalization procedure of $U_{i+1}$.

\begin{figure}
    \centering
    \includegraphics[width=.84\linewidth]{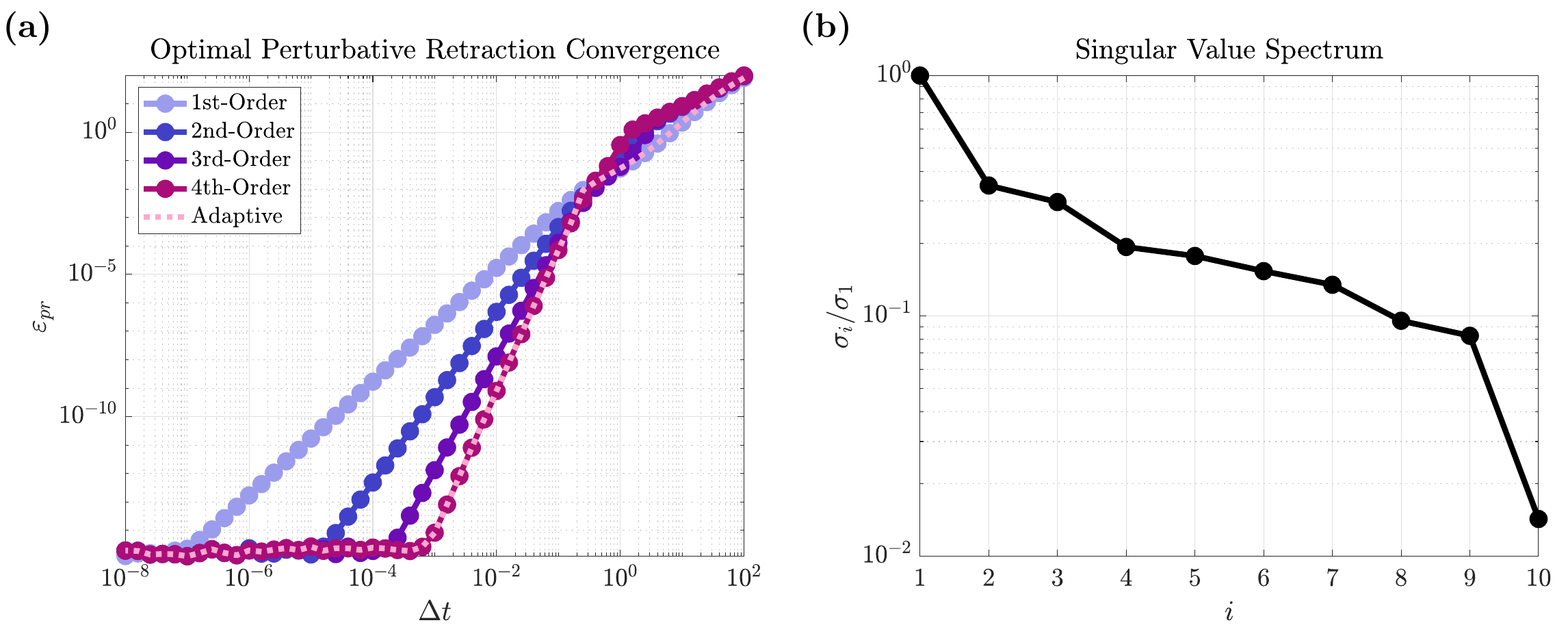}
    \captionsetup{font=footnotesize,labelfont=footnotesize}
    \caption{(a) Our optimal perturbative retractions (\ref{eq:retract_cor}), e.g.\;(\ref{eq:second_order_optimal}), exhibit high-order convergence to the best low-rank approximation. We see linear convergence at rates 2, 3, 4, and 5 in $\epr$ (Table \ref{tab:keyExpressions}). The adaptive algorithm selects which order correction to use by analyzing the convergence of the perturbation series.
    (b) The singular value spectrum of $X_0$, showing singular values span nearly two orders of magnitude. The condition number of $S_0^TS_0$ in the Frobenius norm is 4968.3.}
    \label{fig:optRetConvergence}
\end{figure}

The example problem of matrix addition is constructed as follows. We approximate $X_0 + \Delta t \LLo$ by forming our rank-10 starting point $X_0 \in \mathds{R}^{500\times 220}$; we choose uniformly randomly sampled matrices $U_0 \in \mathds{R}^{500\times 10}$ and $V_0 \in \mathds{R}^{220\times 10}$, and then we orthonormalize them. Then, we choose $S_0\in\mathds{R}^{10\times 10}$, also uniformly randomly sampled, and normalize $S_0$ such that its Frobenius norm is one. Finally, we set $X_0 = U_0S_0V_0^T$. Next, we form a rank-100 matrix $\LLo \in \mathds{R}^{500 \times 220}$ which is constructed by multiplying randomly, uniformly sampled matrices of size $500\times 100$ and $100\times 220$, and then we normalize $\LLo$ to have Frobenius norm one. Our reference solution then becomes $\PM(X_0 + \Delta t \LLo)$, and we compare that with $\Ret_{X_0}(\Delta t \LLo)$.

We illustrate the results in Figures \ref{fig:optRetConvergence} and \ref{fig:optvOriginal}.
We observe that the $k$th-order perturbative retraction exhibits linear convergence to the best low-rank approximation at order $k+1$, where the $+1$ comes from only analyzing the local error instead of the global error after $\BigO(1/\Delta t)$ steps of time integration. In addition, the optimal perturbative retractions are essentially always better than the original perturbative retractions.

\begin{figure}
     \centering
     \includegraphics[width=0.84\textwidth]{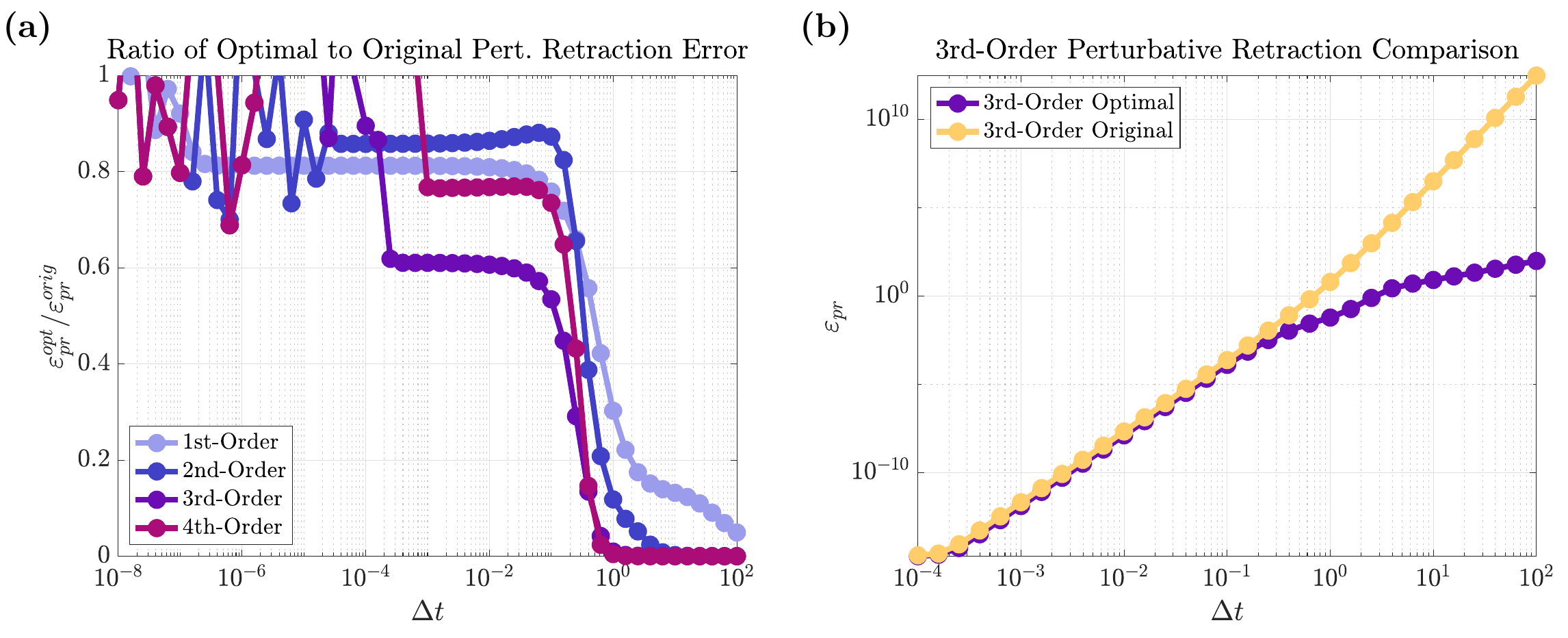}
     \captionsetup{font=footnotesize,labelfont=footnotesize}
     \caption{(a) Ratio of errors of the optimal (\ref{eq:retract_cor}), e.g.\;(\ref{eq:second_order_optimal}), and original \cite{charous_lermusiaux_SJSC2023a} perturbative retractions. The ratio is less than one, indicating the optimal retractions offer an improvement until $\Delta t$ becomes small enough that both retractions' errors are near machine zero, so the ratio fluctuates around one. (b) Explicit comparison of the errors between third-order retractions. The new optimal retraction (\ref{eq:corr3}) is always better, as expected, but performs especially well with large $\Delta t$ by mostly avoiding overshoot \cite{charous_lermusiaux_SJSC2023a}.}
     \label{fig:optvOriginal}
\end{figure}

%-----------------------------------------
%-----------
\section{Gradient-descent retractions}
\label{sec:gradDescent}

%--------------
\subsection{Derivation}
\label{sec:gd-derive}
In this section, we demonstrate how to accelerate convergence to the best low-rank approximation by iteratively applying pre-existing retractions. This methodology generates what we call \say{gradient-descent retractions,} and we show that if the retractions asymptotically approximate the exact projection operator, the gradient-descent retractions converge \emph{superlinearly} to a low-rank matrix $X \in \Mr$ and \emph{linearly} to the best low-rank approximation of a matrix $X \not \in \Mr$.

We begin by recalling the geometric intuition behind Newton's method for minimizing some cost function $J(x)$ in one dimension. By expanding $J(x_k + \Delta x)$ around an iterate $x_k$, we obtain a second-order Taylor series $J(x_k + \Delta x) \approx J(x_k) + \Delta x J'(x_k) + \frac{\Delta x^2}{2} J''(x_k)$. To minimize $J(x_k + \Delta x)$, we differentiate our expression with respect to $\Delta x$ and set it to zero, giving $\Delta x = -J''(x_k)^{-1}J'(x_k)$. In words, we minimized a second-order approximation of the cost function $J$. Alternatively, we note that the first-order optimality constraint is given by $J'(x) = 0$. Newton's method for solving this potentially nonlinear equation is given by Taylor expanding $J'$ in a similar fashion: $J'(x_k + \Delta x) \approx J'(x_k) + \Delta x J''(x_k)$. We set this first-order expression to zero as insisted by our optimality constraint, which yields the same expression for $\Delta x$. In this interpretation, we solve for the zero of a linear approximation of $J'$ \cite{lermusiaux_2.29_notes}.

Now, consider our optimization problem (\ref{eq:opt2}) with optimality condition (\ref{eq:optCondition}). 
By expanding $\Delta t \Uo$ as a perturbation series, implicitly, we are deriving the Taylor series of (\ref{eq:optCondition}).
We may interpret the first-order perturbative retraction as applying one Newton step. Furthermore, our higher-order perturbative retractions may be interpreted then as one step of higher-order gradient-descent iterations generalizing Newton's method. Denoting $\newPoint - X_i$ as the residual between the new point of arbitrary rank and our current point, this interpretation leads to a new iterative algorithm for improving these retractions: simply repeat the retractions applied on the residual, just as one iterates in Newton's method.
Such a gradient-descent retraction scheme is defined in algorithm \ref{alg:gradDescentRets}. It is also natural to consider a retraction where the number of iterations is determined automatically, iterating until some convergence criterion is met. We suggest measuring the Frobenius norm of the difference between iterates of $X_{i+1}$ normalized by the norm of $X_i$; when the norm of the difference drops below a pre-specified value/tolerance, denoted by hyperparameter $\Delta^*$, or we exceed a certain number iterations, denoted by $N_{\text{max}}$, we stop iterating. An automatic gradient-descent retraction is written out in algorithm \ref{alg:autoGradDescentRet}.

\begin{algorithm}
\caption{Fixed-order gradient-descent retraction}
\label{alg:gradDescentRets}
\begin{algorithmic}[1]
\REQUIRE{$X_i \in \Mr$, $\LLo \in \mathds{R}^{m\times n}$, $\Delta t \in \mathds{R}$, $N \in \mathds{Z}^+$}
\ENSURE{$X_{i+1} \in \Mr$}
\STATE{$X_{i+1} = X_i$}
\FOR{$j=1:N$}
\STATE{$X_{i+1} \leftarrow \Ret_{X_{i+1}}(\Delta t \LLo - (X_{i+1} - X_i))$} 
\ENDFOR
\end{algorithmic}
\end{algorithm}

\begin{algorithm}
\caption{Automatic gradient-descent retraction}
\label{alg:autoGradDescentRet}
\begin{algorithmic}[1]
\REQUIRE{$X_i \in \Mr$, $\LLo \in \mathds{R}^{m\times n}$, $\Delta t \in \mathds{R}$, $N_{\text{max}} \in \mathds{Z}^+$, $\Delta^* \in \mathds{R}$}
\ENSURE{$X_{i+1} \in \Mr$}
\STATE{$\alpha = \|X_i\|$}
\STATE{$X_{i+1} = X_i$}
\STATE{$X_{i+1} \leftarrow \Ret_{X_{i+1}}(\Delta t \LLo - (X_{i+1} - X_i))$}
\STATE{$\Delta = \|X_{i+1} - X_i\|$}
\STATE{$j = 1$}
\WHILE{$j < N_{\text{max}}$ and $\Delta/\alpha < \Delta^*$}
\STATE{$j \leftarrow j + 1$}
\STATE{$\hat{X} = X_{i+1}$}
\STATE{$X_{i+1} \leftarrow \Ret_{X_{i+1}}(\Delta t \LLo - (X_{i+1} - X_i))$}
\STATE{$\Delta = \|X_{i+1} - \hat{X}\|$}
\ENDWHILE
\end{algorithmic}
\end{algorithm}

Recall that in Algorithms \ref{alg:gradDescentRets} and \ref{alg:autoGradDescentRet}, whenever $X_{i+1}$ is updated via a retraction, $U_{i+1}$ and $Z_{i+1}$ are updated since $X_{i+1}$ is just shorthand for $U_{i+1}Z_{i+1}^T$. In fact, $X_{i+1}$ is never formed, and doing so may not even be possible on a memory-limited device. In our implementation, we find that using a first-order optimal perturbative retraction (\ref{eq:zeq}, \ref{eq:corr1}) or its robust variant (\ref{eq:zUpdate}, \ref{eq:uUpdateRobust}) (to be discussed in section \ref{sec:robust}) is efficient and extends stability from lemma \ref{lem:stability}, which will be shown next.

Finally, as we consider explicit schemes in this work, we remark that
$\LLo$ can depend on the past iteration points.
This approach is used with perturbations to derive DORK schemes in
section \ref{sec:newDORK}.
In general,
$\LLo$ may also be computed with iterative semi-implicit or implicit schemes and then depend on unknown or future points within $\Delta t$
% Such implicit and iterative schemes are developed in
\cite{charous_PhDThesis2023,charous_lermusiaux_SJSC2023c}.
%--------------
\subsection{Analysis}
\label{sec:analysis}

%\AC{Done. I also changed the j iteration index to be a superscript to be more clear. I think I changed all of them, but please check too in case there are any straggling $X_0$'s or $X_j$ or $X_{j-1}$.}
%\PFJL{This use of j is great now, great idea. One thing: should we also add an iteration index j in Algorithms 3.1 and 3.2? Maybe not as we have the j loop. What about in the Algorithms 4?}
%\AC{I think algorithms 3.1 and 3.2 are good with the iteration index j. Adding a variable with j would suggest defining new variables, but when coding it is better do the operations in-place without wasting more memory. Algorithms 4.1 and 4.2 don't need the j index either since they abstract away the details of the retraction (see line 12 of alg 4.1 and line 3 of alg 4.2). The focus of algorithms 4.1 and 4.2 is not the retractions themselves, but making them rank-adaptive.}

First, we show that the gradient-descent retractions are stable provided that the inner retraction chosen ---- such as an optimal perturbative retraction --- is also stable.
\begin{lemma}
\label{lem:gd_stab}
Let $X_{i+1} = \Ret^{\text{gd}}_{X_i}\left(\Delta t \LLo \right)$ be the result of applying a gradient-descent retraction  from algorithm \ref{alg:gradDescentRets} or \ref{alg:autoGradDescentRet} using an optimal retraction $\Ret^{\text{opt}}$ given by equations (\ref{eq:corr1}-\ref{eq:corr4}) and (\ref{eq:zeq}) as the final inner retraction. Then $\|X_{i+1}\| \leq \|X_i + \Delta t \LLo\|$.
\end{lemma}
\begin{proof}
    Assume we perform $k$ inner iterations of gradient descent. From the last iteration of the gradient-descent retraction, we have that
    \begin{align*}
        \left\|X_{i+1}^{(k)}\right\| &= \left\|\Ret^{\text{opt}}_{X_{i+1}^{(k-1)}}\left(X_i + \Delta t \LLo - X_{i+1}^{(k-1)} \right)\right\|\\
        & \leq \|X_{i+1}^{(k-1)} + X_i + \Delta t \LLo - X_{i+1}^{(k-1)}\|\\
        &= \|X_i + \Delta t \LLo\|
    \end{align*}
    The inequality above follows from lemma \ref{lem:stability}.
\end{proof}
\qed
Hence using the stable, optimal retractions in our gradient-descent retractions retains our stability property. We emphasize that only the final inner retraction in the gradient descent iterations must be stable in order to preserve the overall stability of the retractions.

Next, we analyze the rate of convergence of Algorithm \ref{alg:gradDescentRets} by first assuming that the point we wish to approximate, $X_i + \Delta t \LLo$, is rank-$r$ and then analyzing the case where it is not, i.e., $X_i + \Delta t \LLo \notin \Mr$. Recall that throughout the text, we assume that $\LLo$ is computed explicitly, meaning it may depend on previous states including $X_{i-1}$ and $X_i$ but not unknown or future states such as $X_{i+1}$.
%\PFJL{Here, it can depend on $X_i$ since the integral starts at $t_i$, I made the update (having $\LLo$ depend on $X_i$ also solves a question of Reviewer 2). I would then add a bit here on $r_L$ using what you wrote in the Google doc.}
Letting $r_L$ denote the rank of $\LLo$, we remark that even if $X_i + \Delta t \LLo \in \Mr$, naively computing the sum (without retracting) will result in storing $U_{i+1} \in \mathds{R}^{m\times(r+r_L)}$ and $Z_{i+1} \in \mathds{R}^{n \times (r+r_L)}$ in contrast to $U_{i+1} \in \mathds{R}^{m\times r}$ and $Z_{i+1} \in \mathds{R}^{n \times r}$ after retracting back to $\Mr$. Algorithm \ref{alg:gradDescentRets} provides an efficient way to limit the growth of the rank of $X$ after each time step.

%\PFJL{We should likely restate or note mathematically that $\LLo$ depends on $X_i$ in some nonlinear way (I would do this in the two algorithms or section 3.1, and at the end of section 3.2 as well).}
%\PFJL{We could also do two lemmas, one for rank r and one for not (lemma 3.2 could start from "Starting from (3.2), now without ..." below.} 
%
%\PFJL{After the lemma proofs, we could mention that $\LLo$ could depend on the past $X^{(j-1)}$'s or other points within $\Delta t$, e.g., estimates of $X_{i+1}$, or $X_{i+1/2}$, etc., for iterative implicit (R-K) schemes.}
%\AC{I added a sentence right before the lemma. I also added a sentence at the end of section 3.2 mentioning implicit schemes.}

\begin{lemma}
If $X_i + \Delta t \LLo \in \Mr$, for a small enough $\Delta t$ such that (\ref{eq:suffCond}) holds, iteratively applying retractions that asymptotically approximate $\PM$ to the $p^{\rm th}$ order ($\epr = \BigO(\Delta t^p)$) yields rate-$p$ superlinear convergence to the fixed point $X_i + \Delta t \LLo$\,. \label{lem:geo}
\end{lemma}
\begin{proof}
We wish to approximate some point $X_i + \Delta t \LLo$, and we define an iteration of the following form.
\begin{gather}
    X^{(j)} = \Ret_{X^{(j-1)}}^{(p)}(X_i + \Delta t \LLo - X^{(j-1)}) \label{eq:iter}
\end{gather}
Above, $\Ret^{(p)}$ denotes a retraction that asymptotically approximates $\PM$ to the $p^{\rm th}$ order, and $j$ is our iteration index.
By definition, we then have that
\begin{gather}
    \Ret_{X^{(j-1)}}^{(p)}(X_i + \Delta t \LLo - X^{(j-1)}) = \PM(X_i + \Delta t \LLo) + \BigO\left(\left\|X_i + \Delta t \LLo - X^{(j-1)}\right\|^p\right). \label{eq:retproperty}
\end{gather}
Note that $\PM(X_i + \Delta t \LLo)$ is the fixed point. Let $\epr^{(j)} = X^{(j)} - \PM(X_i + \Delta t \LLo)$ denote the projection-retraction error at iteration $j$. We may write $\|\epr^{(j)}\|$ as a function of $\epr^{(j-1)}$ by combining (\ref{eq:iter}) and (\ref{eq:retproperty}).
\begin{align*}
    \left\|\epr^{(j)}\right\| &= \BigO\left(\left\|X_i + \Delta t \LLo - X^{(j-1)}\right\|^p\right) \\ %\label{eq:convProof}
    &= \BigO\left(\left\|\PM(X_i + \Delta t \LLo) - X^{(j-1)}\right\|^p\right) \nonumber\\
    &= \BigO\left(\left\|\epr^{(j-1)}\right\|^p\right)\nonumber
\end{align*}
Above, we have used the fact that $X_i + \Delta t \LLo \in \Mr$. Hence, we have shown the iteration converges superlinearly at rate $p$ assuming a close enough starting point is chosen such that the sequence converges.
\end{proof}
\qed
Note that a perturbative retraction of order $p-1$ may be used as $\Ret^{(p)}$ since it asymptotically approximates the solution to (\ref{eq:opt1}) and thus asymptotically approximates $\PM$ to the $p^{\rm th}$  order.

If $X_i + \Delta t \LLo \not \in \Mr$, as a first remark, we can expect to see at least linear convergence to the best low-rank approximation by iteratively applying smooth retractions; that is, assuming the iteration converges at all. To see why, consider a new function $g$.
\begin{gather*}
    g(X) \equiv \Ret_X(X_i + \Delta t \LLo - X)
\end{gather*}
Then, we may expand $g(X^{(j-1)})$ 
at $\PM\left(X_i + \Delta t \LLo\right)$ 
via Taylor series.
\begin{equation}
\begin{aligned}
    g(X^{(j-1)}) &= g\left(\PM\left(X_i + \Delta t \LLo\right)\right) \\ &+\left\langle \epr^{(j-1)}, g'\left(\PM \left(X_i + \Delta t \LLo \right)\right)\right\rangle + \BigO\left(\left\|\epr^{(j-1)}\right \|^2\right) \label{eq:linearProof}
\end{aligned}
\end{equation}
This assumes our retractions vary smoothly, which is generally valid except perhaps at singular value crossings (see \cite{feppon_lermusiaux_SIMAX2018a}). 
We note that $g(\PM(X_i + \Delta t \LLo)) = \PM(X_i + \Delta t \LLo)$, and the left-hand side of (\ref{eq:linearProof}) is equal to $X^{(j)}$. Subtracting $\PM(X_i + \Delta t \LLo)$ from both sides gives us $\epr^{(j)}$, and so we have the following.
\begin{equation*}
    \left\|\epr^{(j)}\right\| = \BigO\left(\left\|\epr^{(j-1)}\right\|\right)
\end{equation*}
We have shown that if the iteration converges, it at least converges linearly. 

To show the iteration converges using the above, we would need to show that $\|g'(\PM (X_i + \Delta t \LLo ))\| < 1$, but this is dependent upon the retraction chosen. However, we can instead derive a sufficient condition for convergence. Starting from (\ref{eq:retproperty}), again without assuming $X_i + \Delta t \LLo \in \Mr$, we have that $ \left\|\epr^{(j)}\right\| = \BigO\left(\left\|X_i + \Delta t \LLo - X^{(j-1)}\right\|^p\right)$ and thus
\begin{align*}
    \left\|\epr^{(j)}\right\| \leq \alpha \left\|X_i + \Delta t \LLo - X^{(j-1)}\right\|^p
\end{align*}
for some $\alpha \in \mathds{R}$. 
We can then break the error into different components, denoting the orthogonal complement of the projection operator $\PMp \equiv I - \PM$.
\begin{align*}
   \left\|\epr^{(j)}\right\| &\leq \alpha \left\|\PM\left(X_i + \Delta t \LLo\right) + \PMp\left(X_i + \Delta t \LLo\right) - X^{(j-1)}\right\|^p\\
     &= \alpha\left\|\PMp\left(X_i + \Delta t \LLo\right) - \epr^{(j-1)}\right\|^p\\
     &\leq \alpha \left(\left\|\epr^{(j-1)}\right\| + \left\|\PMp\left(X_i + \Delta t \LLo\right)\right\|\right)^p
\end{align*}
%
%\PFJL{\\
%For the above (lemma 3.2), isn't this still valid:
%$ \left\|\epr^{(j)}\right\| = \BigO\left(\left\|X_i + \Delta t \LLo - X^{(j-1)}\right\|^p\right)
%= \BigO\left(\left\|\PM(X_0 + \Delta t \LLo) - X^{(j-1)}\right\|^p\right)\\~~
%=\left\|\PM\left(X_0 + \Delta t \LLo\right) + \PMp\left(X_0 + \Delta t \LLo\right) - X^{(j-1)}\right\|^p\\~~
%= \left\|\PMp\left(X_0 + \Delta t \LLo\right) - \epr^{(j-1)}\right\|^p$\\~~
%and if yes, couldn't we use this to get rid of the alpha? and so get a better/cleaner condition than (3.4)?\\}
%\AC{No, you cannot drop the BigO as you have done. The alpha comes from dropping the BigO.}
%  
Recall that $\PMp\left(X_i + \Delta t \LLo\right)$ is related to the normal closure error $\en$, the distance between the best low-rank approximation and $X_i + \Delta t \LLo$. We make the assumption that this value is small, otherwise, the low-rank approximation would be inaccurate. So, it is reasonable to insist that 
%
%\PFJL{I am wondering if we can say something better than this:}
%
%\AC{If you can, great. I spent a lot of time trying to get something better but couldn't. This is not my strength, so you may be able to do better.}
%\PFJL{I will think some more. One idea: if we use singular values, $\epr^{(j-1)}$ is a difference of terms with sing values larger than $\sigma^{(j-1)}_{r+1}$ while the integrated orthogonal project term $\PMp$ is of the order
%$\sigma^{(j-1)}_{r+1} + \BigO\left(\Delta t\right)$. Hence, it is the $\BigO\left(\Delta t\right)$ terms that could grow larger, so we could do a condition on that?}
%\AC{Yeah, there may be something there. Let me know if you make any progress on that.}
\begin{gather}
    \left\|\PMp\left(X_i + \Delta t \LLo\right)\right\| < \left(\frac{\left\|\epr^{(j-1)}\right\|}{\alpha}\right)^{\frac{1}{p}} - \left\|\epr^{(j-1)} \right\| \label{eq:suffCond},
\end{gather}
where $\|\epr^{(j-1)}\|$ can also be made arbitrarily small by choosing a point sufficiently close to $\PM(X_i + \Delta t \LLo)$. This can easily be accomplished, no matter the dynamics $\LLo$, by refining the time step $\Delta t$. 
Hence, if (\ref{eq:suffCond}) is satisfied (e.g., $\en$ decays fast enough with $\Delta t$), our iteration converges since we then have shown $\|\epr^{(j)}\| < \|\epr^{(j-1)}\|$. 
We interpret this sufficient condition by considering the singular values of the quantities in (\ref{eq:suffCond}). 
Specifically, $\|\epr^{(j-1)}\| = \BigO(\sigma^{(j-1)}_1)$ where $\sigma^{(j-1)}_{1}$ denotes the largest singular value of $\epr^{(j-1)}$, 
i.e.\;the dominant dynamics not yet captured in $X^{(j-1)}$,
and $\|\PMp\left(X_i + \Delta t \LLo\right)\|=\BigO(\sigma^{(i+1)}_{r+1})$, where
$\sigma^{(i+1)}_{r+1}$ is the largest singular value of the projection of the dynamics over $\Delta t$ in the orthogonal complement space (recall $X_i$ is of rank $r$). As long as $\sigma^{i+1}_{r+1} \ll \sigma^{(j-1)}_{m}$, the iteration will converge; if not, the total error will be dominated by 
the error in the normal space (see figure \ref{fig:rankDiscovery}).
For an in-depth analysis of the decay of this $\PMp\left(X_i + \Delta t \LLo\right)$, we refer to \cite{iterconv}. In summary, high-order convergence is preserved if the initial point is close enough to the fixed point, and that distance is less than the smallest singular value of the fixed point. If the smallest singular value of the fixed point is negligible, e.g., of the order of machine precision, the rank of the system should be truncated, which is discussed in section \ref{sec:rankred}. 

%\AC{I don't think the following paragraph fits here. Maybe we don't need it since we mention this in the introduction.}

%\AC{Again, 2023c is not submitted yet, so not sure if we can cite.}
%\PFJL{Yes, I added your thesis. I assume we will submit the implicit-dork paper by the time we get the proofs for this one.}

%-------------------
\subsection{Illustration}
\label{grad-descent-ex}

To demonstrate the efficacy of the gradient-descent and optimal perturbative retractions, we study their convergence properties. We use the same problem setup as in Section \ref{sec:optimal-pert-ex}, hence $i=0$.

\begin{figure}[h]
     \centering
     \includegraphics[width=0.85\textwidth]{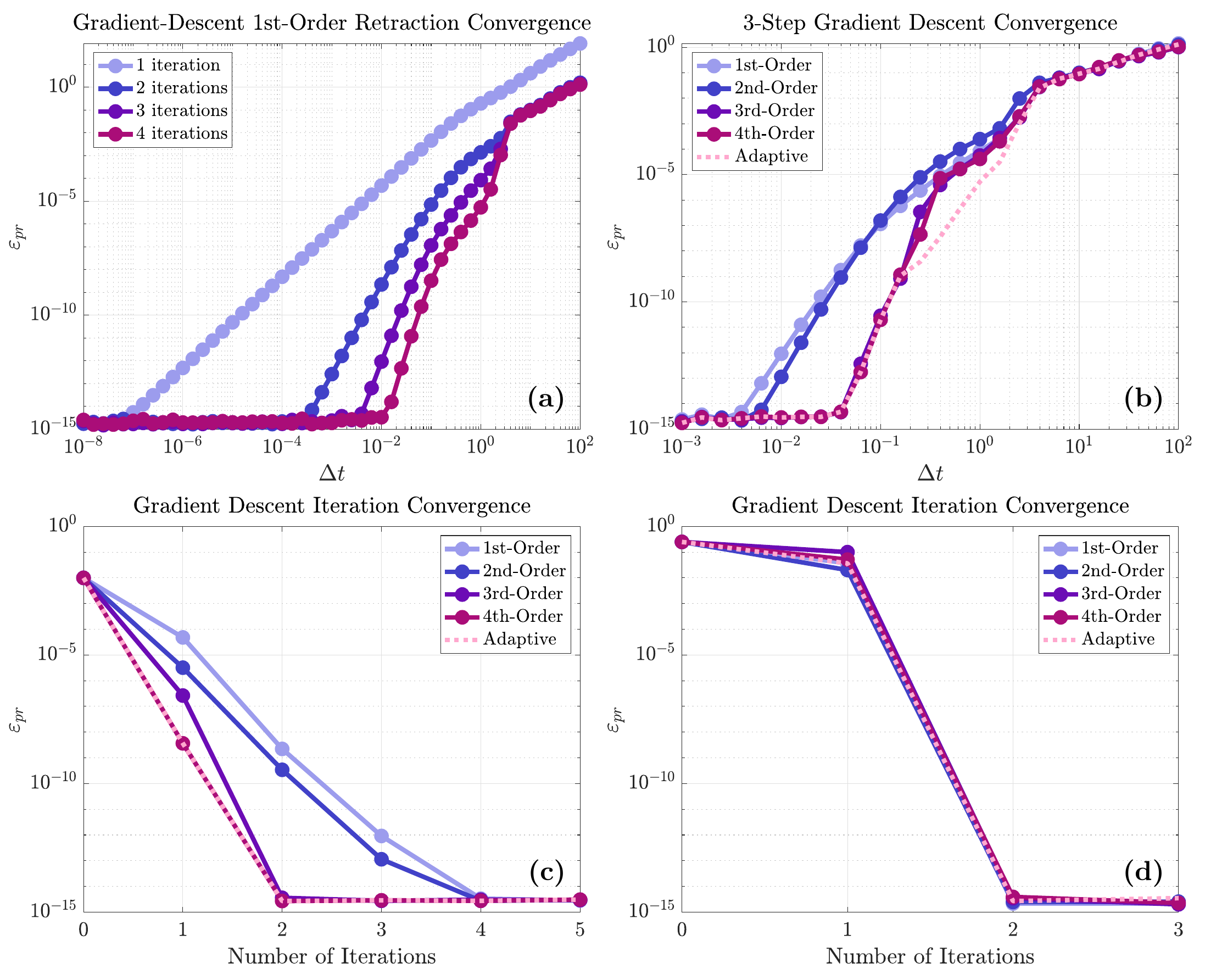}
     \captionsetup{font=footnotesize,labelfont=footnotesize}
     \caption{(a) Convergence of the gradient-descent iterations with various $\Delta t$ using a first-order perturbative retraction (\ref{eq:corr1}). Using more iterations provides more accurate approximations. 
     (b) Convergence using 3 steps of gradient descent with different retractions; high-order (and adaptive-order \cite{charous_lermusiaux_SJSC2023a}) retractions lead to faster convergence. 
     (c)
     Projection-retraction error vs.\ iteration number using different retractions with $\Delta t = 0.01$. Again, the higher-order retractions converge the fastest. 
    (d)
     Projection-retraction error vs.\ iteration number with $\Delta t = 0.25$, but here we insist $X_0 + \Delta t \LLo' \in \Mr$. All retractions converge to the best approximation to machine precision in two iterations.}
     \label{fig:gradDescentConvergence}
\end{figure}

In Figure \ref{fig:gradDescentConvergence}, the first three panels are analogous to those of figures \ref{fig:optRetConvergence} and \ref{fig:optvOriginal}, where $X_0 + \Delta t \LLo \not \in \Mr$. By just using first-order retractions iteratively, we obtain high-order convergence to the best low-rank approximation. Hence, the computational cost for high-order retractions does not grow due to the larger expressions for higher-order corrections. In addition, we observe that utilizing higher-order perturbative retractions with multiple iterations accelerates convergence further. In the bottom right panel, we swap out $\LLo$ for $\LLo'$ such that $X_0 + \Delta t \LLo' = \PM(X_0 + \Delta t \LLo)$. In this case, even with a relatively large $\Delta t$, we observe extremely rapid convergence to the best low-rank approximation, which suggests (but does not rigorously show) superlinear convergence.

%-----------------------------------------
%--------------
\section{Robust retractions with rank adaptation}
\label{sec:rankAdaptive}

\subsection{Derivation}
\label{sec:rankAdaptiveDerivation}
So far, we have shown how to approximate the truncated SVD via a perturbation series, yielding novel, effective retractions accurate up to arbitrary order. We have also shown how to iterate on the low-rank manifold to perform arbitrarily-high-order gradient descent on the low-rank manifold using these retractions. In this section, we develop an algorithm to perform rank-adaptive integration, which requires a retraction that is robust to small singular values.

To see why a robust retraction for rank-adaptive integration is needed, consider the first step at time $t_{i+1}$ after augmenting the rank of our system from rank $r_i$ to $r_{i+1}$. We denote the singular values of $X_i$ as $\sigma_i^{(j)}$, $j=1,\ldots, r_i$, and those of $X_{i+1}$ as $\sigma_{i+1}^{(j)}$, $j=1,\ldots, r_{i+1}$. 
By equation (\ref{eq:bestProjMethod}), for the best possible approximation $X_{i+1}^* = \PMrii(X_i + \Delta t \LLo)$, we then have by Weyl's inequality \cite{weyl},
\begin{gather*}
    \sigma_{i+1}^{(j)} \leq \sigma^{(j)}_{i} + \BigO\left(\Delta t\right).
\end{gather*}
This, of course, assumes that $\|\LLo\| = \BigO(1)$. The implications are that for $r_i < j \leq r_{i+1}$, as $\sigma^{(j)}_{i}=0$, we have $\sigma_{i+1}^{(j)} \sim \Delta t$, and so $X_{i+1}$ will be ill-conditioned assuming $\|X_{i+1}\| = \BigO(1)$. As such, we must avoid inverting $Z_i^TZ_i$ in our retractions.

\subsubsection{Robust retractions}
\label{sec:robust}
Here, we develop two methods to robustify our high-order retractions. The most straightforward approach is as follows: instead of inverting $Z_i^TZ_i$, we simply use the Moore-Penrose pseudoinverse. 
When $Z_i^TZ_i$ is rank-deficient, the linear systems (\ref{eq:corr1}-\ref{eq:corr4}) do not admit unique solutions for $\{\Ud_j\}_j$. As such, we may regularize the system and still satisfy the equations. So that the perturbation series converges, it is natural to find least-squares solutions, which is precisely what the pseudoinverse does.
The pseudoinverse of any given matrix $A$ may be calculated by taking the truncated SVD (or eigendecomposition), $USV^T$, and then inverting singular values with magnitude greater than some threshold. Since we always invert a Hermitian matrix, a more numerically accurate algorithm takes the truncated SVD of $A$ and then the QR decomposition of $\sqrt{S}U^T$; the transpose of the resulting upper triangular matrix $R$ gives a Cholesky decomposition of $A = R^TR$. 
To solve a rank-$r$ linear system, we can then simply use forward and backward substitution. Otherwise, in the rank-deficient case, a linear least-squares problem is solved.

The advantages of this pseudoinverse approach are that it can be used with both our high-order perturbative retractions and incorporated into our gradient-descent retractions, and it preserves mode continuity, which is desirable for uncertainty quantification \cite{lin_PhDThesis2020,lin_lermusiaux_NM2021}.  
%\cite{lin_PhDThesis2020,lin_lermusiaux_NM2021,lin_lermusiaux_2022prep2}. 
The pseudoinversion, however, induces an additional source of error accumulation. When the singular values of $X_i$ are small but nonzero, truncating the singular values of the correlation matrix is inherently another approximation, and the truncation threshold we choose generates a lower bound on the minimum error a low-rank integration scheme can attain. 
This effect is depicted in Section \ref{sec:numExp},  Fig.\ \ref{fig:matDiffEq}.

To avoid the error accumulation due to the pseudoinverse, we introduce an alternative robustification technique by building off of our first-order retraction (\ref{eq:corr1}) before re-orthonormalization.
\begin{gather}
    U_{i+1} = U_i + \Delta t \PUpi \LLo Z_i\left(Z_i^T Z_i \right)^{-1} \label{eq:uNonRobust}\\
    Z_{i+1}^{\text{optimal}} = Z_iU_i^TU_{i+1} + \Delta t \LLo^TU_{i+1} \label{eq:zUpdate}
\end{gather}
Importantly, we only care about the column space of $U_{i+1}$ rather than its individual columns because $Z_{i+1}$ will optimally adapt to any rotation to $U_{i+1}$. So, we right-multiply our update equation for $U$ by $Z_i^TZ_i$ to obtain the following equation.
\begin{gather}
    U_{i+1}\left(Z_i^T Z_i \right) = U_i\left(Z_i^T Z_i \right) + \Delta t \PUpi \LLo Z_i \label{eq:uzUpdate}
\end{gather}
To reiterate, the left-hand sides of equations (\ref{eq:uNonRobust}) and (\ref{eq:uzUpdate}) span the same space. So, our robust first-order retraction is obtained by just orthonormalizing (\ref{eq:uzUpdate}) and then computing $Z_{i+1}$ according to (\ref{eq:zUpdate}).
\begin{gather}
     U_{i+1}^{\text{robust}} = \texttt{orth}\left(U_i\left(Z_i^T Z_i \right) + \Delta t \PUpi \LLo Z_i \right) \label{eq:uUpdateRobust}
\end{gather}
The orthonormalization may be done by a QR decomposition, SVD, or other algorithm of choice such as the re-orthonormalization procedure in \cite{lin_lermusiaux_NM2021}.

Now we have a robust, first-order perturbative retraction. It would be nice to obtain robust higher-order corrections for Eqs.\ (\ref{eq:corr2}, \ref{eq:corr3}, \ref{eq:corr4}). 
However, the above methodology does not extend to these higher-order corrections. This is because we need to sum the corrections in (\ref{eq:pertSeries}), which is not invariant to individual rotations of $\Ud_j$. Fortunately, we can utilize our gradient-descent iterations (Alg.\ \ref{alg:gradDescentRets}) with the robust first-order retraction (Eqs.\ \ref{eq:zUpdate} and \ref{eq:uUpdateRobust}) that utilizes a re-orthonormalization step in each inner gradient-descent iteration to obtain arbitrarily high-order retractions robust to overapproximation provided $X_i + \Delta t \LLo \in \Mr$ or $\PMp(X_i + \Delta t \LLo)$ decays sufficiently quickly (see section \ref{sec:analysis}). Note, however, that high-order convergence of the projection-retraction error is not guaranteed, and a counterexample when the normal closure error $\en$ dominates is provided in figure \ref{fig:rankDiscovery}.

The robust retraction defined by Eqs.\ (\ref{eq:zUpdate}) and (\ref{eq:uUpdateRobust}) avoids the error accumulation from the pseudoinverse. But, we pay for it by sacrificing mode continuity. In (\ref{eq:uUpdateRobust}), we are integrating the span of $U$ rather than individual modes, and so we do not keep track of how each mode evolves. In summary, there is a tradeoff, and the user can decide what is important for a particular problem. The pseudoinverse results in additional error accumulation but preserves mode continuity. 
The new robust retraction with gradient descent does not accumulate extra error, but it loses mode continuity. 
Both techniques preserve high-order convergence when the system is well-conditioned. However, to preserve high-order convergence when the system is ill-conditioned, time steps must be sufficiently small such that the distance from our initial point $X_i$ and the fixed point $\PM(X_i + \Delta t \LLo)$ is smaller than the smallest singular value of $X_i$ \cite{iterconv}. Geometrically, the low-rank manifold curvature tends to infinity, and these high-order corrections cannot capture areas of the manifold that are not smooth unless sufficiently small steps are taken.

\subsubsection{Rank augmentation}
\label{sec:rankaug}
As mentioned in Section \ref{sec:intro}, different algorithms and criteria have been developed to determine when to augment the rank of our system \cite{feppon_lermusiaux_SIMAX2018a,gao2022riemannian,lin_PhDThesis2020,zhou2016riemannian}. Our approach, depicted in Figure \ref{fig:adaptiverank}, is to measure the angle between the full-rank direction $\LLo$ and its projection onto the tangent space. 
This provides a non-dimensional parameter related to how fast the dynamics departs from the low-rank manifold. If the angle is larger than some threshold $\theta^*$, we increment the rank by a pre-specified value $r_{\text{inc}}$ up to a maximum rank of $r_{\text{max}}$. 
To reduce rapid variations in the rank as a function of time, one can alter the rank of the solution only if not altered in recent time steps. Simple dynamical systems can be used to do this using a memory and forgetting rate, e.g., \cite{lermusiaux_PhDThesis1997,lermusiaux_JMS2001,ryu_et_al_Oceans2021}.

\begin{figure}[h]
    \centering
    \includegraphics[width=.48\linewidth]{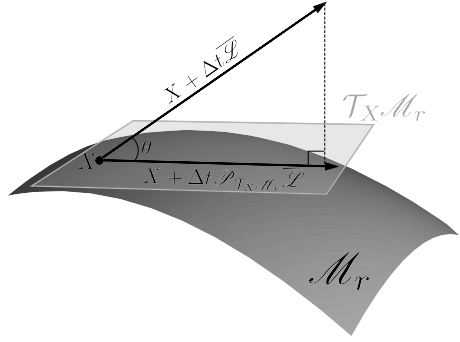}
    \captionsetup{font=footnotesize,labelfont=footnotesize}
    \caption{Depicted above is a point $X$ on the low-rank manifold $\Mr$. From that point, the dynamics dictate we travel in direction $\LLo$; this direction is not known a priori and changes within time integration steps, which our DORK schemes take into account. To stay on the manifold, the dynamics (to first-order) must be projected onto the tangent space. The angle $\theta$ between $\LLo$ and $\PT\LLo$ is used as our metric to determine when to increase the rank of the solution.}
    \label{fig:adaptiverank}
\end{figure}

The next question becomes: how do we actually augment the rank of the system? First, we artificially augment $X_i = U_iZ_i^T$ as follows.
\begin{gather}
    X_i = \begin{bmatrix}U_i & Q \end{bmatrix} \begin{bmatrix}Z_i^T \\ 0 \end{bmatrix} \equiv \hat{U}_i\hat{Z}_i^T \label{eq:augmented}
\end{gather}
Above, $Q$ is chosen to be semi-orthonormal with columns that are orthogonal to the columns of $U_i$. Then, we may simply apply our retractions using $\hat{U}_i$ and $\hat{Z}_i$. However, we note that if we apply our update to $U_{i+1}$ first, we will not update the subspace spanned by $Q$ in $\hat{U}$ due to the zeros in $\hat{Z}$. As such, we recommend first updating $\hat{Z}$ via a simplified form of (\ref{eq:zUpdate}) below. 
\begin{gather}
    \hat{Z}_i \leftarrow \hat{Z}_i + \Delta t \LLo^T \hat{U}_i\label{eq:simpleZUpdate}
\end{gather}
To compensate for the change in $\hat{Z}_i$, we must calculate $\LLo'$ such that $U_iZ_i^T + \Delta t \LLo = \hat{U}_i\hat{Z}_i^T + \Delta t \LLo'$. Finally, we apply a retraction of our choice starting from point $\hat{U}_i\hat{Z}_i^T$ in direction $\LLo'$ to update the span of $\hat{U}$ (in addition to updating $\hat{Z}$ again).

The ambiguity here is how to choose $Q$, i.e.\ how to augment our subspace $U$. In \cite{lin_PhDThesis2020}, Lin considers taking the first few left singular vectors of the residual of $\LLo$, defined as $(I-\PT) \LLo$, with the largest singular values \cite{lin_PhDThesis2020}. 
This approach, while principled, can be expensive as it can require computing the SVD of $(I-\PT) \LLo$ at each step. A computationally efficient implementation of \cite{lin_PhDThesis2020}'s approach is to use the randomized SVD. 
Assuming $\LLo$ is low-rank, obtaining the first few singular vectors of $(I-\PT) \LLo$ is computationally cheap and provably accurate (see \cite{randSVD}). 
The randomized SVD inherently introduces stochastic variability into the rank augmentation algorithm, which is analyzed in \cite{hauck2022predictor}. Here, the randomized SVD is only used to choose the basis with which we augment the subspace, so it does not immediately induce a new error in integrating the dynamical system along the low-rank manifold. However, the stochastic variability may result in a slightly suboptimal choice of the basis $Q$. 
See Algorithm \ref{alg:rankAugment} for an implementation. We note that any retraction robust to overapproximation may be used in Algorithm \ref{alg:rankAugment}, e.g.\ the robust perturbative retraction with automatic gradient descent, and it should ensure that $U_{i+1}$ is orthonormalized. 

\subsubsection{Rank reduction} \label{sec:rankred} To reduce the rank of a point after integration, the process is straightforward \cite{lermusiaux_PhysD2007,feppon_lermusiaux_SIREV2018}. Because $U_{i+1}$ is orthonormal, the singular values of $X_{i+1}$ are defined by the eigenvalues of $Z_{i+1}^TZ_{i+1}$. As such, the impact of truncating the rank of $X_{i+1}$ measured by the Frobenius norm is equal to the sum of the eigenvalues that we omit. Hence, we drop the smallest eigenvalues and their associated eigenvectors so long as the truncation results in a change in Frobenius norm less than a predetermined value $\sigma^*$. To remove the respective singular vectors from $X_i$, we project $U_{i+1}$ and $Z_{i+1}$ onto a truncated set of eigenvectors of $Z_{i+1}^TZ_{i+1}$.
We refer to Algorithm \ref{alg:rankAugment} for details.

\begin{algorithm}
\caption{Angular rank-adaptive retraction with randomized SVD}
\label{alg:rankAugment}
\begin{algorithmic}[1]
\REQUIRE{$X_i \in \Mri$, $\LLo \in \mathds{R}^{m\times n}$, $\Delta t \in \mathds{R}$, $\theta^* \in [0,\pi/2]$, $\sigma^* \in [0,1]$, $r_{\text{inc}} \in \mathds{Z}^+$, $r_{\text{max}} \in \mathds{Z}^+$}
\ENSURE{$X_{i+1} \in \Mrii$}
\STATE{$r_i = \text{rank}\left(X_i \right)$}
\STATE{$\Dot{U}_i = \PUnp \LLo Z_i (Z_i^TZ_i)^{-1}$, \quad $\Dot{Z}_i = \LLo^TU_i$}
\STATE{$\PT \LLo = U_i \Dot{Z}_i + \Dot{U}_iZ_i^T$}
\STATE{$\theta = \cos^{-1}\left(\frac{\|\PT \LLo\|}{\|\LLo\|}\right)$}
\IF{(($\theta^* = 0$) or ($\theta > \theta^*$)) and ($r_i < r_{\text{max}}$)}
\STATE{$r_{\text{inc}} = \min(r_i,r_{\text{inc}},r_{\text{max}}-r_i)$}
\STATE{$Q = \texttt{randSVD}\left(\PUnp \LLo \right)$ via \cite[alg. 4.4]{randSVD}}
\STATE{$\hat{U}_i = \begin{bmatrix}U_i & Q \end{bmatrix}, \quad Z_i = \begin{bmatrix} Z_i & \texttt{zeros}\left(n,r_{\text{inc}}\right) \end{bmatrix}$}
\STATE{$\hat{Z}_i \leftarrow \hat{Z}_i + \Delta t \LLo^T\hat{U}_i$}
\STATE{$\LLo \leftarrow \LLo - (\hat{U}_i\hat{Z}_i^T - U_iZ_i^T)/\Delta t$}
\ENDIF
\STATE{$X_{i+1} = \Ret_{X_i}(\Delta t \LLo - (X_{i+1} - X_i))$}
\STATE{$V,\Lambda = \texttt{eig}\left(Z_i^TZ_i \right)$ with eigenvalues (and associated eigenvectors) sorted in ascending order}
\STATE{$\text{cutoff} = \text{last index } j \text{ where } (\sum_{i=1}^j\Lambda(i)/\sum_i\Lambda(i))^{1/2} < \sigma^*$}
\STATE{$V \leftarrow V(:,\text{cutoff}+1:\text{end})$}
\STATE{$U_{i+1} \leftarrow U_{i+1}V, \quad Z_{i+1} \leftarrow Z_{i+1}V$}
\end{algorithmic}
\end{algorithm}

\subsection{Illustration}
\label{sec:rankAdaptIllustration}
To illustrate our robust retraction with gradient descent and rank augmentation, we consider the problem of rank discovery. Suppose we have a low-rank approximation $X$ to $\Xf$, where $X \in \mathds{R}^{50{}0\times 220}$ has rank 20, and $\Xf = X + \Delta t \LLo$ has rank 125, where $X$ and $\LLo$ have Frobenius norm one, and $\Delta t = 0.1$. The matrices $X$ and $\LLo$ are initialized randomly as in Sect.\ \ref{sec:optimal-pert-ex}.

Our goal is to recover the full-rank $\Xf$ efficiently using our retractions. We combine our rank-adaptive retraction \ref{alg:rankAugment} and our automatic gradient-descent retraction \ref{alg:autoGradDescentRet} to create an iterative scheme that keeps increasing the rank of $X$ until the normalized local retraction error $\el$ has norm below a specified threshold, $\el^*$. We specify the details of the implementation in Algorithm \ref{alg:rankDiscovery}, where we set the incremental rank increase to 25, the max rank to 200, $\el^* = 10^{-6}$, and the maximum number of iterations in gradient descent $N_{\text{max}} = 16$.

\begin{algorithm}
\caption{Automatic, adaptive retraction for rank discovery and compression}
\label{alg:rankDiscovery}
\begin{algorithmic}[1]
\REQUIRE{$X_i \in \Mri$, $\LLo = U_L Z_L^T \in \mathds{R}^{m\times n}$, $\Delta t \in \mathds{R}$, $N \in \mathds{Z}^+$, $\Delta^* \in \mathds{R}$, $r_{\text{inc}} \in \mathds{Z}^+$, $r_{\text{max}} \in \mathds{Z}^+$, $\el^* \in \mathds{R}$, $N_\text{max} \in \mathds{Z}^+$}
\ENSURE{$X_{i+1} \in \Mrii$}
\STATE{$X_{i+1} = X_i$}
\WHILE{$\|X_{i}+\Delta t \LLo - X_{i+1}\|/\|X_i\| > \el^*$}
\STATE{$X_{i+1} = \texttt{rankAdaptiveRetraction}(X_{i+1},\LLo - (X_{i+1}-X_i)/\Delta t = U_LZ_L^T,\Delta t, \theta^* = 0, \sigma^* = \el^*, r_{\text{inc}}, r_{\text{max}}, N_{\text{max}})$ using automatic gradient-descent retraction internally with $\Delta^* = \el^*$} \ENDWHILE
\end{algorithmic}
\end{algorithm}

Figure \ref{fig:rankDiscovery} shows different errors from Table \ref{tab:keyExpressions} as a function of iteration number as well as how the rank changes. When we underestimate the rank of the solution in the first few outer steps of the algorithm, the local retraction error and the normal closure error stay relatively large during the inner gradient-descent steps since they include error in the normal space. We see the projection-retraction error decrease, albeit quite slowly, when we underapproximate the rank. This is in contrast to what we observed in section \ref{sec:gradDescent}, where we rapdily converged to the best low-rank approximation. The difference here is that we are starting without a good initial guess for the subspace spanned by $U$ after augmenting the rank. Essentially, we are trying to search a high-dimensional vector space via perturbations, so convergence is slow. What's more, in optimizing the span of $U$, there exist local minima for the left singular vectors of $\Xf$ with associated singular values that are not the largest in magnitude. For example if we set $U$ to the left singular vectors with the smallest singular values, we would still reach a stationary point where $\PT \LLo = 0$. The saving grace of this algorithm, though, is that when we hit or exceed the rank of the full-rank solution, all of the errors plumet due to the quadratic convergence of the projection-retraction error and the fact that $\en = 0$.

\begin{figure}
     \centering
     \includegraphics[width=1\textwidth]{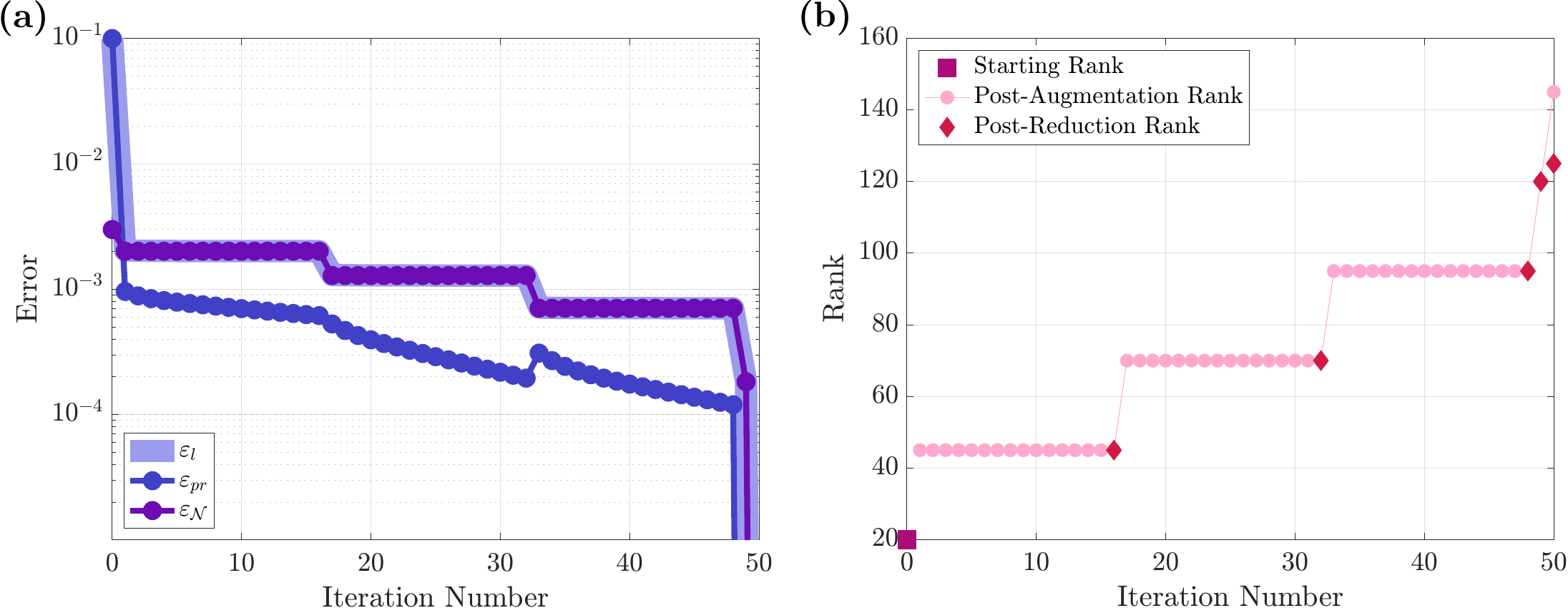}
     \captionsetup{font=footnotesize,labelfont=footnotesize}
     \caption{Rank discovery using algorithm \ref{alg:rankDiscovery} with $\el^* = 10^{-6}$. 
     (a) Local retraction error, projection-retraction error, and normal closure error (see Table \ref{tab:keyExpressions}). During the internal iterations of the gradient-descent retraction, the projection-retraction decreases. Then, when the rank is increased, the projection-retraction error increases since we are now measuring the error between our current approximation and new, higher-rank matrix. The normal closure error remains constant during the inner gradient-descent iterations since it is not a function of the retractions. Of course, it monotonically decreases as we increase the rank. The local retraction error remains almost constant during the inner iterations of gradient descent since it includes $\en$, which dominates the total error. 
     (b) Rank of each iteration as a result of augmentation as well as truncation. Note that the truncation only takes effect once we exceed the true rank of the solution at the last iteration.}
     \label{fig:rankDiscovery}
\end{figure}

%-----------------------------------------
%--------------
\section{Novel DORK schemes}
\label{sec:newDORK}

In this Section, we build off our optimal perturbative and gradient-descent retractions to develop two new low-rank integration schemes based off the Dynamically Orthogonal Runge-Kutta (DORK) schemes \cite{charous_lermusiaux_SJSC2023a}. These schemes integrate along the low-rank manifold, updating the subspace upon which we project the system's dynamics in stages during the time step. 
To do so, instead of taking $\LLo$ to be a constant $\BigO(1)$ term, we expand it as a perturbation series itself.
\begin{gather}
    \Delta t \LLo \equiv \sum_{j=1}^{\infty}\Delta t^{j}\Ld_j \label{eq:LPert}
\end{gather}
We denote the partial sum $\LLo^{(k)} = \sum_{j=1}^k \Delta t^j \Ld_j$, which is a $k$th-order integration scheme, and each $\Ld_j$ is a $j$th-order correction to $\LLo^{(j-1)}$ such that $\LLo^{(j)}$ increases an order of accuracy. To construct integration schemes as a perturbation series, we refer to \cite{charous_lermusiaux_SJSC2023a}.

%----------
\subsection{so-DORK schemes}

First, we derive the stable, optimal Dynamically Orthogonal Runge-Kutta (so-DORK) schemes.
%\textbf{a}daptive-rank \textbf{d}ynamically \textbf{o}rthogonal \textbf{R}unge-\textbf{K}utta st\textbf{able} (ADORKABLE) schemes. 
The process is straightforward; we substitute (\ref{eq:pertSeries}) and (\ref{eq:LPert}) into (\ref{eq:optCondition}) and solve for the first- through fourth-order corrections.
\begingroup
\allowdisplaybreaks
\begin{subequations}
\begin{align}
    &\Ud_1 = \PUnp \Ld_1 Z_i (Z_i^TZ_i)^{-1} \label{eq:corr1_so-DORK}\\
    &\Ud_2 = \left[\PUnp\left(\Ld_1\Ld_1^TU_i + \Ld_2 Z_i\right) - \Ud_1\left(U_i^T \Ld_1 Z_i +  Z_i^T \Ld_1^T U_i\right) \right](Z_i^TZ_i)^{-1} \label{eq:corr2_so-DORK}\\
    &\begin{aligned}
    &\Ud_3 = \left[\PUnp \left(\Ld_1 \Ld_1^T \Ud_1 + \left(\Ld_1\Ld_2^T + \Ld_2 \Ld_1^T\right) U_i + \Ld_3 Z_i \right) \right.\\&\phantom{\Ud_4} -
    \Ud_2 \left(U_i^T \Ld_1 Z_i +  Z_i^T \Ld_1^T U_i\right) -
    \Ud_1 \left(U_i^T \Ld_2 Z_i +  Z_i^T \Ld_2^T U_i \right. \\&\phantom{\Ud_4} \left.\left. + U_i^T\Ld_1\Ld_1^T U_i + \Ud_1^T\Ld_1 Z_i + Z_i^T\Ld_1^T\Ud_1 - \Ud_1^T\Ud_1 Z_i^T Z_i  \right)\right](Z_i^T Z_i)^{-1} \label{eq:corr3_so-DORK}
    \end{aligned}\\
    &\Ud_4 = \left[\PUnp\left(\Ld_1\Ld_1^T\Ud_2 + \left(\Ld_1\Ld_2^T + \Ld_2 \Ld_1^T\right) \Ud_1 + \left(\Ld_1\Ld_3^T + \Ld_3 \Ld_1^T + \Ld_2\Ld_2^T\right) U_i \right. \right. \notag
    \\ &\phantom{\Ud_4}\left. + \Ld_4 Z_i \right) - \Ud_3\left(U_i^T \Ld_1 Z_i +  Z_i^T \Ld_1^T U_i\right) \notag
     -\Ud_2\left(U_i^T \Ld_2 Z_i +  Z_i^T \Ld_2^T U_i \right. \notag
     \\&\phantom{\Ud_4}\begin{aligned}
     &\left. + U_i^T\Ld_1\Ld_1^T U_i + \Ud_1^T\Ld_1 Z_i + Z_i^T\Ld_1^T\Ud_1 - \Ud_1^T\Ud_1 Z_i^T Z_i  \right) -\Ud_1\left(U_i^T \Ld_3 Z_i  \right.\\ 
    &\left. +  Z_i^T \Ld_3^T U_i + U_i^T\left(\Ld_2\Ld_1^T + \Ld_1\Ld_2^T\right)U_i + \Ud_1^T\Ld_2 Z_i + Z_i^T\Ld_2^T\Ud_1 \right. \label{eq:corr4_so-DORK}
    \end{aligned}\\
    &\phantom{\Ud_4}+ U_i^T\Ld_1\Ld_1^T\Ud_1 + \Ud_1\Ld_1\Ld_1^T U_i + \Ud_2^T\Ld_1 Z_i + Z_i^T\Ld_1^T\Ud_2 - \Ud_1^T\Ud_2 Z_i^TZ_i \notag \\ 
    &\phantom{\Ud_4}- \left. \left. \Ud_2^T\Ud_1 Z_i^T Z_i - \Ud_1^T\Ud_1 U_i^T\Ld_1 Z_i - \Ud_1^T\Ud_1Z_i^T\Ld_1^TU_i \right)\right](Z_i^T Z_i)^{-1} \notag
\end{align}
\label{eq:so-DORK}
\end{subequations}
\endgroup

We first remark that the additional high-order $\Ld_j$ are not in equations (\ref{eq:corr1}-\ref{eq:corr4}) as they then remain within $\LLo$.
Second, the update equation for the coefficients $Z_{i+1}$ is (\ref{eq:zUpdate}) or (\ref{eq:zeq}) which leads to stability properties.
The following lemma follows directly from lemma \ref{lem:stability}.
\begin{lemma}
Starting from $X_i$, let $X_{i+1}$ be the result of integrating $\Delta t \LLo$ with an so-DORK scheme given by equations (\ref{eq:corr1_so-DORK}-\ref{eq:corr4_so-DORK}) and (\ref{eq:zeq}). Then $\|X_{i+1}\| \leq \|X_i + \Delta t \LLo\|$. \label{lem:stabilityDORK}
\end{lemma}
\begin{proof}
Let $X_{i+1} = U_{i+1}Z_{i+1}^T$. By construction, $U_{i+1} \in \Stmr$, so $\|X_{i+1}\| = \|Z_{i+1}\|$. 
By (\ref{eq:zeq}), $\|Z_{i+1}\| = \|(X_{i} + \Delta t \LLo)^T U_{i+1}\| \leq \|X_{i} + \Delta t \LLo\|$.
\end{proof}
\qed
This implies that if the Runge-Kutta integration scheme of $\LL(X,t;\omega)$ is stable, its so-DORK counterpart will also be stable, making so-DORK special. Note that full-rank stability integrating $\LL(\Xf,t;\omega)$ does not guarantee low-rank stability due to the dynamical model closure error $\ed$ and normal closure error $\en$ (see table \ref{tab:keyExpressions}). 
However, in practice, these errors are small if the DLRA is expected to perform well; as a result, they tend not to affect stability.

\subsection{gd-DORK schemes}

Now, we utilize our gradient-descent retractions from Section \ref{sec:gradDescent} and develop
gradient-descent Dynamically Orthogonal Runge-Kutta
(gd-DORK) schemes. 
%\textbf{d}ynamically \textbf{o}rthogonal \textbf{R}unge-\textbf{K}utta gradient-\textbf{descent} (DORK descent) schemes. 
This family of integration schemes also updates the subspace as we integrate, but instead of approximating the high-order curvature of the low-rank manifold from one fixed starting point, the manifold is locally approximated at points along our trajectory. 
A $k$th-order integration scheme is given in Algorithm \ref{alg:DORKdescent}; note that we incorporate perturbations to $\LLo$ as we integrate. For computational efficiency, we typically use a first-order retraction in each substep, though any order retraction may be used. 
Furthermore, an automatic variant may also be derived from Algorithm \ref{alg:autoGradDescentRet}. 
Finally, 
we recall that even for gd-DORK, $X_{i+1}$ is not formed; instead, $U_{i+1}$ and $Z_{i+1}^T$ are updated during the gradient descent inner iterations, as explained at the end of section \ref{sec:gd-derive}.

\begin{algorithm}
\caption{gd-DORK Schemes}
\label{alg:DORKdescent}
\begin{algorithmic}[1]
\REQUIRE{$X_i \in \Mr$, $\{\LLo^{(j)} \in \mathds{R}^{m\times n}\}_{j=1}^k$, $\Delta t \in \mathds{R}$}
\ENSURE{$X_{i+1} \in \Mr$}
\STATE{$X_{i+1} = X_i$}
\FOR{$j=1:k$}
\STATE{$X_{i+1} \leftarrow \Ret_{X_{i+1}}(\Delta t \LLo^{(j)}- (X_{i+1} - X_i))$}
\ENDFOR
\end{algorithmic}
\end{algorithm}

\begin{lemma}
Given a close enough starting point and $\|X_i\| \sim \|\LLo^{(j)}\|$ for all $j$, gd-DORK schemes are consistent, $k$th-order accurate integration schemes. \label{lem:dorkdescent}
\end{lemma}
\begin{proof}
We will show that, starting from a point $X_i \in \Mr$, Algorithm \ref{alg:DORKdescent} asymptotically approximates $\PM(X_i + \Delta t \LLo^{(k)})$. We will proceed inductively, starting with one initial retraction.

Let $X^{(1)} = \Ret^{(p)}_{X_i}\left(\Delta t \LLo^{(1)}\right)$, where $\Ret^{(p)}$ is a retraction that asymptotically approximates $\PM$ to the $p$th order, and $p \geq 2$. Assuming that $\|X_i\| \sim \|\LLo^{(j)}\|$ for all $j$, we have:
\begin{gather*}
    X^{(1)} = \PM\left(X_i + \Delta t \LLo^{(1)} \right) + \BigO(\Delta t^p).
\end{gather*}
On iteration $j$, we will assume that we have $X^{(j-1)} = \PM(X_i + \Delta t \LLo^{(j-1)}) + \BigO(\Delta t^q)$, where $q \geq j$. Certainly for $j = 2$, this is true. We will analyze the error of another step of gradient descent of sec.\ \ref{sec:gradDescent}, using (\ref{eq:iter}), (\ref{eq:retproperty}), and the assumption about $X^{(j-1)}$.
\begingroup
\allowdisplaybreaks
\begin{align}
    X^{(j)} &= \Ret_{X^{(j-1)}}^{(p)}\left(\Delta t \LLo^{(j)} + (X_i - X^{(j-1)}) \right) \nonumber \\
    &= \PM\left(X_i + \Delta t \LLo^{(j)} \right) + \BigO\left(\left\|\Delta t \LLo^{(j)} + (X_i - X^{(j-1)})\right\|^p\right) \nonumber\\
    &\begin{multlined}= \PM\left(X_i + \Delta t \LLo^{(j)} \right) \\+ \BigO\left(\left\|\Delta t \LLo^{(j)} + X_i - \PM(X_i + \Delta t \LLo^{(j-1)}) + \BigO\left(\Delta t^q\right)\right\|^p\right) \end{multlined} \label{eq:DORKDescentError}
\end{align}
\endgroup

Next, we decompose the full-rank point we are approximating, $X_i + \Delta t \LLo^{(j)}$, into two components, and we then use the fact that $\LLo^{(j)}$ is a $\BigO(\Delta t ^j)$ correction to $\LLo^{(j-1)}$.
\begin{align}
    X_i + \Delta t \LLo^{(j)} &= \PM\left(X_i + \Delta t \LLo^{(j)} \right) + \PMp\left(X_i + \Delta t \LLo^{(j)} \right) \nonumber \\
    &= \PM\left(X_i + \Delta t \LLo^{(j-1)} \right) + \BigO\left(\Delta t^j\right)  + \PMp\left(X_i + \Delta t \LLo^{(j)} \right) \label{eq:pointDecomposed}
\end{align}
We now substitute (\ref{eq:pointDecomposed}) into (\ref{eq:DORKDescentError}).
\begin{gather}
    X^{(j)} = \PM\left(X_i + \Delta t \LLo^{(j)} \right) + \BigO\left(\left\|\BigO\left(\Delta t^j, \Delta t^q\right) + \PMp\left(X_i + \Delta t \LLo^{(j)} \right)\right\|^p\right) \label{eq:DORKDescentConvergence}
\end{gather}
Letting $\epr^{(j)}$ denote the projection-retraction error at iteration $j$ and recalling that $q \geq j$, if $X_i + \Delta t \LLo^{(j)} \in \Mr$, we have that $\|\epr^{(j)}\| \sim \Delta t^{jp}$. Otherwise, we refer to the proof in \cite{iterconv} that shows the orthogonal error decays provided that we start from a point with a distance to the fixed point $\PM(X_i + \Delta t \LLo^{(j)})$ less than the smallest singular value of the fixed point. Hence by induction, we have shown that the gd-DORK scheme is consistent and preserves the order of convergence of the full-rank integration scheme.
\end{proof}
\qed

Analogously to the gradient-descent retractions, if we use stable, optimal retractions as our inner retraction, the resulting gd-DORK scheme will be stable, following directly from lemma \ref{lem:gd_stab}.

\begin{lemma}
\label{lem:gd_dork_stab}
Starting from $X_i$, let $X_{i+1}$ be the result of integrating $\Delta t \LLo$ with a gd-DORK scheme given by algorithm \ref{alg:DORKdescent} using an optimal retraction $\Ret^{\text{opt}}$ given by equations (\ref{eq:corr1}-\ref{eq:corr4}) and (\ref{eq:zeq}) as the final inner retraction. Then $\|X_{i+1}\| \leq \|X_i + \Delta t \LLo\|$.
\end{lemma}
\begin{proof}
    Assume we perform $k$ inner iterations of gradient descent. From the last iteration of the gradient-descent retraction, we have that
    \begin{align*}
        \left\|X_{i+1}^{(k)}\right\| &= \left\|\Ret^{\text{opt}}_{X_{i+1}^{(k-1)}}\left(X_i + \Delta t \LLo^{(k)} - X_{i+1}^{(k-1)} \right)\right\|\\
        & \leq \|X_{i+1}^{(k-1)} + X_i + \Delta t \LLo^{(k)} - X_{i+1}^{(k)}\|\\
        &= \|X_i + \Delta t \LLo^{(k-1)}\|
    \end{align*}
    The inequality above follows from lemma \ref{lem:stability}, and recall that $\LLo = \LLo^{(k)}$ for a $k$th-order accurate integration scheme.
\end{proof}
\qed

As a final remark, we note a key difference between so-DORK and gd-DORK that affects the algorithmic implementation and is linked to (\ref{eq:LPert}).
The so-DORK schemes operate on updates $\Ld_j$ to the differential operator $\LLo^{(j-1)}$ while gd-DORK schemes operate on the $j$th-order integrator $\LLo^{(j)}$ itself. As such, unless updates to $\LLo^{(j)}$ can be computed along the way (as in Heun's method), the most accurate $\LLo^{(j)}$ available at each gradient-descent step should be used. 
That is, for higher-order Runge-Kutta methods where points and functions must be evaluated outside of the integration path, it is more accurate to pre-compute a high-order approximation to $\LLo$ rather than increment the integration order one-by-one as in Alg.\ \ref{alg:DORKdescent}.
%

%\PFJL{We should add here a few sentences para on how (or how not) U and Z are computed with sg-DORK and how or not the U and optimal Z could be computed more frequently, and the possible sgd-DORK scheme. See my comment before.}
%\AC{I think I addressed this with a new paragraph at the end of 3.1.}

%-----------------------------------------
%--------------
\section{Numerical experiments}
\label{sec:numExp}
\subsection{Matrix differential equation}
\label{sec:matDiffEq}

First, we compare the efficacy of our new DORK schemes to other integration schemes in the literature on a system of linear oscillators. As in \cite[sec.\ 7.2]{charous_lermusiaux_SJSC2023a}, we solve for a rank-16 approximation of a $26\times 26$ system governed by $\ddot{\Xf} = -\Omega^2 \Xf$ with initial conditions $\Xf(0) = \Xf_0$, $\dot{\Xf}(0) = \dot{\Xf}_0$). We select $\Omega = \diag(\omega_1,\omega_1,\omega_2,\omega_2,\ldots,\omega_{13},\omega_{13})$ with each $\omega_i$ independently chosen from a standard normal distribution. We construct a matrix $R(t)\in\mathds{R}^{26\times 26}$ as a tridiagonal filled with $2\times 2$ block rotation matrices $R_1,\ldots,R_{13}$ such that,
\begin{gather*}
    R_i(t) = \begin{bmatrix}\cos(\omega_i t) & -\sin(\omega_i t) \\ \sin(\omega_i t) & \cos(\omega_i t) \end{bmatrix},
\end{gather*}
and then $R(t) = \diag(R_1,\ldots,R_{13})$. In contrast to the example in \cite{charous_lermusiaux_SJSC2023a}, $S\in\mathds{R}^{26\times 26}$ is a diagonal matrix with non-increasing entries such that the singular values of the system are $S_{ii} = 100 + 10z_i$ for $1 \leq i \leq 14$ and $S_{ii} = 10^{-5(1+(i-14)/12)}$ for $15 \leq i \leq 26$, where $z_i$ are realizations of independent, standard normal random variables. By construction, our system is ill-conditioned with two small singular values out of the sixteen that are not truncated; the condition number of the rank-16 truncated version of $S^TS$ in the Frobenius norm is approximately $3.00 \cdot 10^{15}$. We construct a new matrix $Q\in\mathds{R}^{26\times 26}$ by orthonormalizing a matrix with independent samples from a uniform distribution. It is simple to verify that the exact solution is given by $\Xf(t) = R(t)QS$ with $\Xf_0 = R(0)QS$ and $\dot{\Xf}_0 = \Dot{R}(0)QS$. To apply our DORK schemes, we integrate an augmented state $\begin{bmatrix} X & \dot{X} \end{bmatrix}$.

Figure \ref{fig:matDiffEq} shows the error measured in the Frobenius norm, normalized by the norm of the initial conditions. Compared to the original DORK schemes, so-DORK and gd-DORK are more accurate; they reduce the rate at which error is accumulated, and they preserve high-order convergence that is eliminated by the ill-conditioned matrix inversion in the original DORK schemes. We note however that so-DORK does not converge all the way to the best approximation. This is due to error accumulated by the pseudoinverse. In selecting the threshold at which we truncate singular values before inverting our correlation matrix, we must balance the error incurred by inverting a singular matrix and the error incurred by approximating our correlation matrix. The gd-DORK scheme avoids this by using the robust scheme (\ref{eq:uUpdateRobust}, \ref{eq:zUpdate}) in each gradient descent step. This entirely avoid matrix inversion at the cost of losing mode continuity.

\begin{figure}
     \centering
     \includegraphics[width=\textwidth]{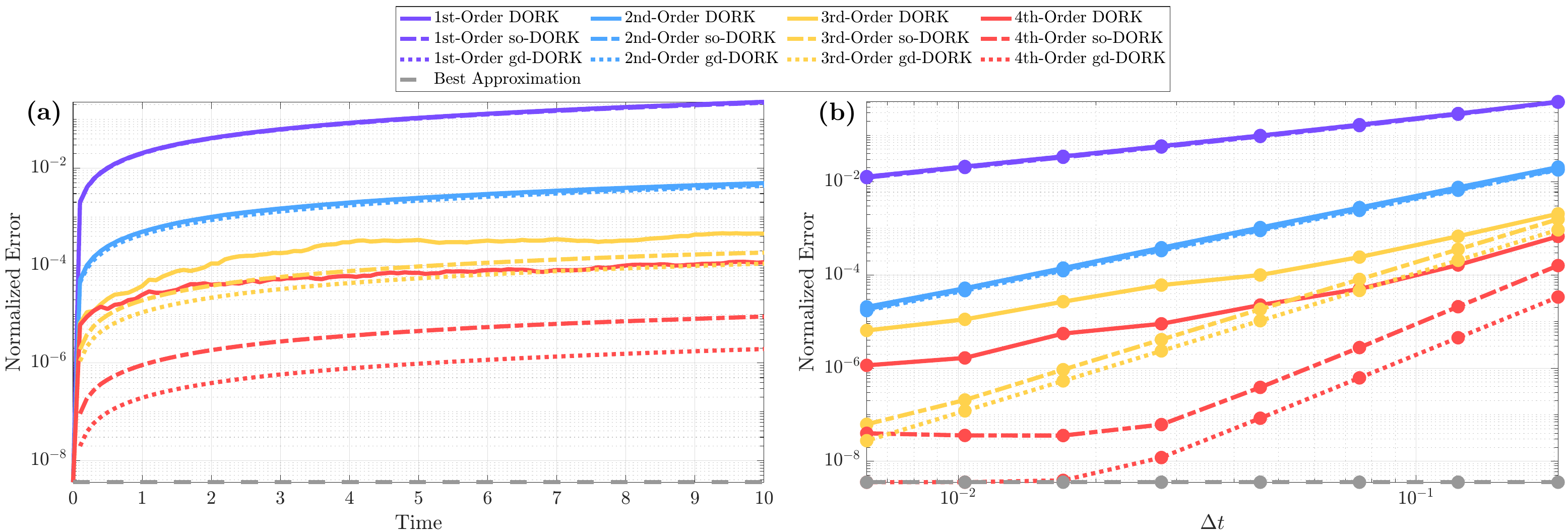}
     \captionsetup{font=footnotesize,labelfont=footnotesize}
     \caption{Matrix differential equation.
     (a) Error accumulation vs.\ time for twelve numerical integration schemes for $t\in [0,10]$ and $\Delta t = 0.1$. As a threshold for truncating singular values for pseudoinversion, we use $10^{-9}\|X(t)\|$. The so-DORK and gd-DORK methods outperform classic DORK, especially at high order.
     (b) Convergence of the error at $t=10$ for multiple values of $\Delta t$. The so-DORK and gd-DORK schemes preserve high-order convergence even when the system is ill-conditioned.}
     \label{fig:matDiffEq}
\end{figure}

In addition, we compare the error accumulated by the three DORK schemes to state-of-the-art numerical schemes.
%in the literature. 
Namely, we implemented the projected Runge-Kutta (PRK) schemes \cite{rkRets} and the symmetrized projector-splitting integrator \cite{projSplit}. All integration schemes are of second-order accuracy, and results are displayed in Table \ref{tab:MatError}. Overall, we see that all of the DORK schemes improve the projector-splitting and projected Runge-Kutta methods as they update the subspace onto which we project the system dynamics between time steps. Furthermore, the so-DORK and gd-DORK methods outperform classic DORK.

\begin{table}
\small
\centering
\begin{tabular}{| c |c c c c c|} 
 \hline
$N_t$ & PRK &  Proj.-Split. & DORK & so-DORK & gd-DORK\\
 \hline
$50$ &  $2.11 \cdot 10^{-2}$ & $2.11 \cdot 10^{-2}$ & $2.06 \cdot 10^{-2}$ & $1.86 \cdot 10^{-2}$ & $1.80 \cdot 10^{-2}$\\
$134$  & $2.86 \cdot 10^{-3}$ &  $2.86 \cdot 10^{-3}$ &  $2.80 \cdot 10^{-3}$ & $2.63 \cdot 10^{-3}$ & $2.43 \cdot 10^{-3}$\\
$968$  &  $5.40 \cdot 10^{-5}$ &  $5.40 \cdot 10^{-5}$ &  $5.20 \cdot 10^{-5}$ & $4.99 \cdot 10^{-5}$ & $4.62 \cdot 10^{-5}$\\
 \hline
\end{tabular}
\normalsize
\captionsetup{font=footnotesize,labelfont=footnotesize}
\caption{Matrix differential equation. Frobenius norms of the final errors after integrating for $t\in[0,10]$ with the projected Runge-Kutta (PRK), projector-splitting, DORK, so-DORK, and gd-DORK second-order integrators, for three values of the number of time steps, $N_t$. DORK schemes are more accurate by incorporating dynamic subspace updates as we integrate, and so-DORK and gd-DORK are most accurate.}
\label{tab:MatError}
\end{table}

%--------------------------
\subsection{Advection-diffusion partial differential equation}
\label{sec:advDiff}

As our next example, we solve the 2D linear advection-diffusion PDE,
\begin{gather}
    \frac{\partial u}{\partial t} + c \cdot \nabla u = \nu \nabla^2 u. \label{eq:advDiffPDE}
\end{gather}
We impose periodic boundary conditions on a $[0,1]^2$ domain. 
This domain
is discretized in space by $256$ points in each dimension using a second-order central finite difference scheme. For time $t \in [0,5]$, to emphasize curvature-related errors, we employ $\Delta t = 10^{-3}$ with a variety of time-integration schemes. For our velocity field $c$, we consider potential flow with some background horizontal velocity superposed with a Rankine vortex.
\begin{gather*}
    c_x = \frac{4}{3} + \begin{cases}
     -\frac{\Gamma y_o}{R^2} & \sqrt{x_o^2 + y_o^2} < R\\
    -\frac{\Gamma y_o}{x_o^2 + y_o^2} & \sqrt{x_o^2 + y_o^2} \geq R
    \end{cases}, \quad
    c_y = \begin{cases}
     \frac{\Gamma x_o}{R^2} & \sqrt{x_o^2 + y_o^2} < R\\
    \frac{\Gamma x_o}{x_o^2 + y_o^2} & \sqrt{x_o^2 + y_o^2} \geq R
    \end{cases}
\end{gather*}
We choose $\Gamma = 8R/3$, $R = 1/16$, and offset variables $x_o = x - 0.5$, $y_o = y - 0.5$. Our initial conditions are haphazardly placed Gaussians,
\begin{gather}
    u(x,y,t=0) = \sum_{i=1}^{7}a_i \exp\left({-\frac{(x-\mu^{(x)}_i)^2 + (y-\mu^{(y)}_i)^2}{2\sigma_i^2}}\right), \label{eq:initCondit}
\end{gather}
where 
\begin{gather*}
    a = \begin{bmatrix}1, & 1, & 1, & 2, & 1, & 1, & 1 \end{bmatrix}^T\\
    \mu^{(x)} = \begin{bmatrix}0.5,   & 0.3,   & 0.354,  & 0.4,   & 0.15,   & 0.65,   & 0.35 \end{bmatrix}^T\\
    \mu^{(y)} = \begin{bmatrix}0.5, & 0.5, & 0.5, & 0.35, & 0.65, & 0.25, & 0.65 \end{bmatrix}^T\\
    \sigma^2 = \begin{bmatrix}200, & 500, & 400, & 400, & 400, & 450, & 450 \end{bmatrix}^T.
\end{gather*} 
To efficiently evolve the system, we truncate the rank of $c_x$ and $c_y$ to a rank-4 approximation, which is the threshold where singular values are less than $10^{-1}$ times the largest singular value, and we use these approximate velocity fields for both the full-rank and low-rank simulations so that we may directly compare the integration error between numerical schemes. In Figure \ref{fig:advDiff}, we show full-rank and low-rank solutions computed with Heun's method and the second-order so-DORK scheme, respectively. We fix the rank to $r = 5$ to contrast with our rank-adaptive solution which dynamically adjusts the rank to capture feature solutions on the fly. Clearly, the rank-adaptive solution is preferable to the under-resolved rank-5 solution. The rank initially increases to accommodate for the advected field hitting the vortex, but then the rank decreases as diffusion dominates the system. The adaptive-rank algorithm \ref{alg:rankAugment} tracks the rank of the full-order solution quite well. Adjusting $\sigma^*$ and $\theta^*$ could yield an even closer match at an increased computational cost.

\begin{figure}
     \centering
     \includegraphics[width=\textwidth]{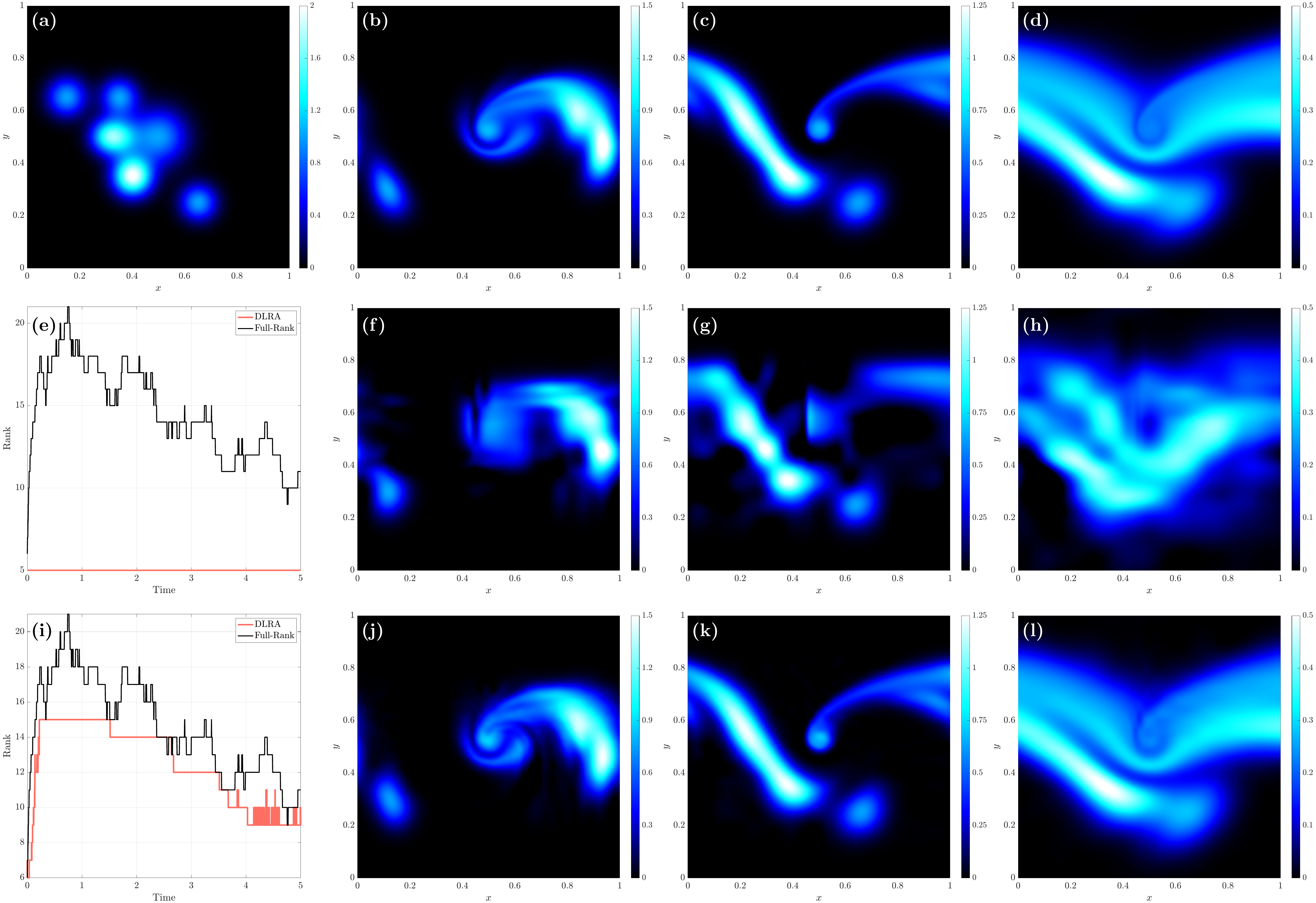}
     \captionsetup{font=footnotesize,labelfont=footnotesize}
     \caption{Advection-diffusion PDE. (a) Initial conditions (\ref{eq:initCondit}). 
     (b-d) Full-rank solution at $t = 0.3$, $0.6$, and $3.0$, respectively. 
     (e) Rank as a function of time: fixed rank-5 (red) and rank of the full-order simulation thresholded at $\sigma^* = 2 \cdot 10^{-3}$ (black).   
     (f-h): Rank-5 solution at the aforementioned three times. 
     (i) As (e), but for the \emph{adaptive} rank as a function of time with $\theta^* = 0.1$, same $\sigma^* = 2 \cdot 10^{-3}$, $r_{\text{inc}} = 1$, and $r_{\text{max}} = 20$ (red).
     (j-l) Corresponding rank-adaptive solution at the same three times.
     The rank-adaptive solution captures the full-rank features while eliminating redundant degrees of freedom.}
     \label{fig:advDiff}
\end{figure}

We also integrate the system with the same integrators from the literature as in Section \ref{sec:matDiffEq}. The mean error (as measured by the deviation from the full-rank numerical solution) divided by the norm of the initial conditions is enumerated in Table \ref{tab:AdvectionDiffError} at different ranks. At lower ranks, the DORK and so-DORK schemes perform the best, possibly because they directly incorporate the high-order curvature of the low-rank manifold, and the system is not ill-conditioned, meaning the perturbation series converge quickly. At higher ranks, the classic DORK scheme does not converge due to the singular correlation matrix. This is where the gd-DORK scheme excels since in each substep, it only projects onto the tangent space of the low-rank manifold while still updating the subspace as we integrate.

\begin{table}
\small
\centering
\begin{tabular}{| c |c c c c c|} 
 \hline
$r$ & PRK &  Proj.-Split. & DORK & so-DORK & gd-DORK\\
 \hline
$5$ &  $2.20 \cdot 10^{-1}$ & $2.13 \cdot 10^{-1}$ & $2.07 \cdot 10^{-1}$ & $2.09 \cdot 10^{-1}$ & $2.14 \cdot 10^{-1}$\\
 \hline
$15$  &  $7.68 \cdot 10^{-3}$ & $7.61 \cdot 10^{-3}$ & DNC & $1.04 \cdot 10^{-2}$ & $7.53 \cdot 10^{-3}$\\
 \hline
\end{tabular}
\normalsize
\captionsetup{font=footnotesize,labelfont=footnotesize}
\caption{ 
%This table lists 
Normalized, time-averaged errors from integrating (\ref{eq:advDiffPDE}) with second-order accurate low-rank schemes in space and time. At rank-5, the DORK and so-DORK are most accurate possibly due to their high-order approximation of the low-rank manifold curvature. At rank-15, the system is ill-conditioned, so the classic DORK scheme does not converge. The so-DORK scheme incurs slightly more error due to the pseudoinversion, and the gd-DORK scheme is most accurate.}
\label{tab:AdvectionDiffError}
\end{table}

\subsection{Stochastic Fisher-KPP PDE}
\label{sec:fisherKPP}
In this section, we showcase the flexibility of our gradient-descent retractions by employing them in an implicit-explicit (IMEX) numerical integration scheme rather than an explicit Runge-Kutta scheme. We solve the stochastic Fisher-KPP equation, a reaction–diffusion PDE that, for example, describes biological population or chemical reaction dynamics with diffusion \cite{fisher,kpp}.
\begin{gather}
    \frac{\partial u}{\partial t} = D\frac{\partial^2 u}{\partial x^2} + r(\omega) u(1-u) \label{eq:kpp}
\end{gather}
The PDE (\ref{eq:kpp}) is stochastic (S-PDE) because we model the reaction rate $r$ as a random variable $r \sim \mathcal{U}[1/4,1/2]$, where $\omega$ denotes a simple event in the event space $\Omega$. In general, (\ref{eq:kpp}) admits traveling waves as solutions and is used in ecology \cite{eco,fkppeco}, biology \cite{bio}, and more \cite{reacdiff}. It also forms a building block of more complex ocean biogeochemical models \cite{besiktepe_et_al_JMS2003,lermusiaux_et_al_Oceanog2011,gupta_et_al_Oceans2019,gupta_lermusiaux_PRSA2021}.

Presently, we consider the domain $x \in [0,40]$ and time duration $t \in [0,12.5]$. For boundary conditions, we assume zero diffusive fluxes at the two boundaries,
\begin{gather}
    \frac{\partial u}{\partial x}(0,t) =0~ , \quad \frac{\partial u}{\partial x}(40,t) = 0 . \nonumber
\end{gather}
The initial conditions are stochastic, of the following form.
\begin{gather*}
    u(x,0;\omega) = a(\omega)e^{-b(\omega)x^2}, \quad a \sim \mathcal{U}[1/5,2/5], \quad b \sim \mathcal{U}[1/10,11/10]
\end{gather*}
The random variables $a$, $b$, $r$ are all independent, and we set the diffusion coefficient $D = 1$. 

To solve the S-PDE numerically, we employ two approaches: Monte Carlo (MC) realizations and a DLRA of the MC realizations, where each column of our solution matrix at a fixed time represents one stochastic realization. We use an implicit-explicit Crank-Nicholson leapfrog time integration scheme \cite{imex}, treating the diffusion term with Crank-Nicholson and the reaction term explicitly with leapfrog. Such an integration scheme is numerically efficient because we separate the deterministic and stochastic operators; by treating the stochastic operator explicitly, we avoid having to invert a different matrix for each column/realization of the Monte Carlo solution. 
To initialize the numerical integration, as the leapfrog scheme requires the solution at 2 previous time steps, we use a first-order time-integration scheme for the first time step; the diffusion term is then treated implicitly and the reaction term explicitly. For the numerical space-time discretization, we use $1{,}000$ grid points in $x$, $10{,}001$ points in time, and a second-order centered finite difference stencil for the diffusion term. For the stochastic discretization, we use $1{,}000$ MC realizations.

\begin{figure}
     \centering
     \includegraphics[width=\textwidth]{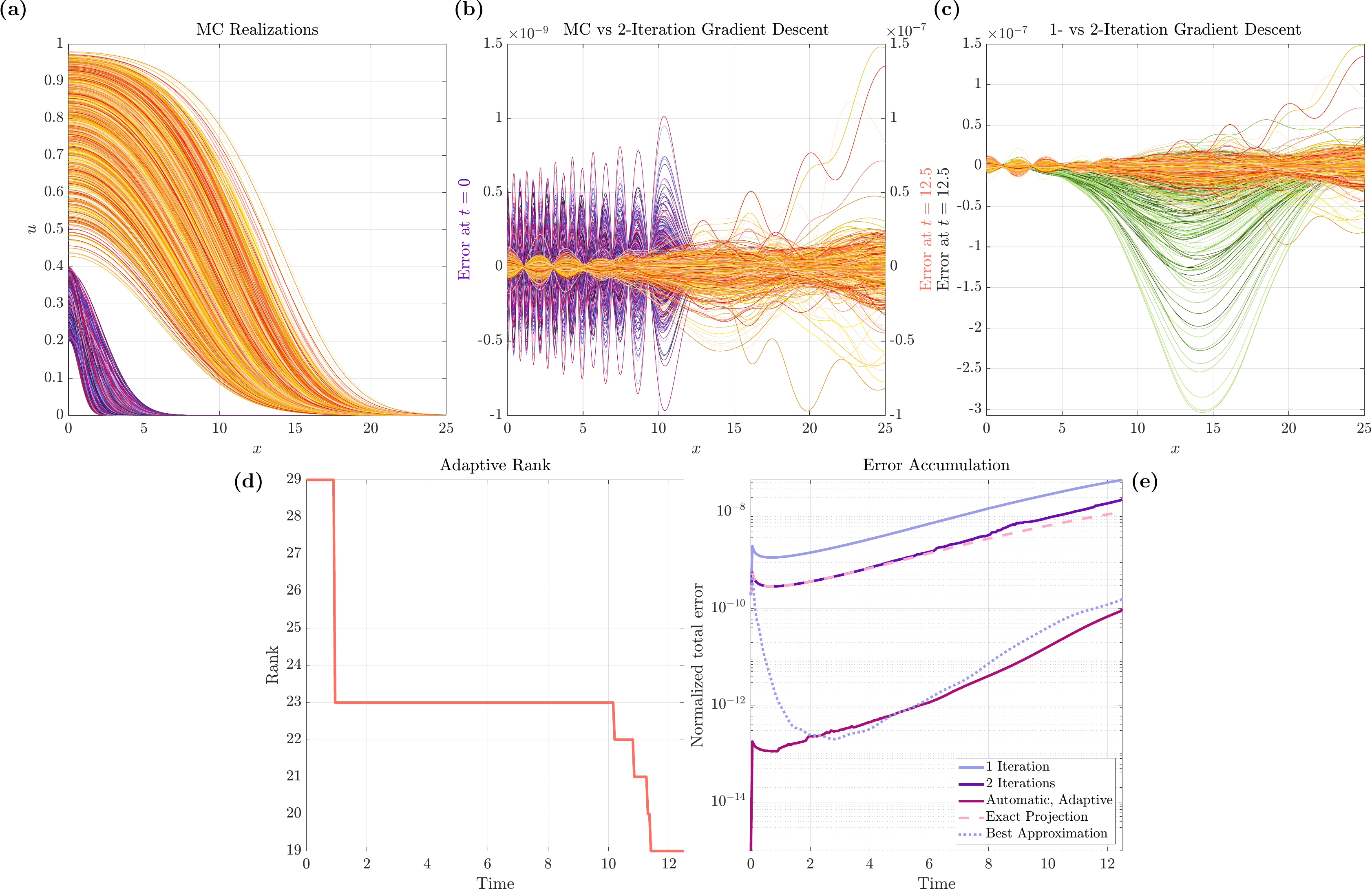}
     \captionsetup{font=footnotesize,labelfont=footnotesize}
     \caption{
     Stochastic Fisher-KPP PDE
     (a) Reference solution, the Monte Carlo runs, at two different times. The cool (pink-purple) colors show solution realizations $u(x,t;\omega)$ at the initial conditions $t=0$, and the warm (yellow-orange) colors show solution realizations at $t = 12.5$. 
     (b) Error of each reconstructed MC realization at the initial and  final times; in both instances, errors are centered around zero and much smaller than the solution magnitudes $u(x,t;\omega)$ (error scales at $t=0$ and $t = 12.5$ differ, and both differ from the scales for $u(x,t;\omega)$). 
     (c) Compares the errors between the 1-iteration (green) and 2-iteration (yellow-orange, as in (b)) gradient-descent retractions at the final time step.
     (d) Rank as a function of time for the automatic, rank-adaptive integration scheme.
     (e) Error measured in the Frobenius norm, normalized by the Frobenius norm of the MC solution at final time. From this non-dimensional error, we find that the 2-iteration gradient-descent retraction more closely follows the error incurred by the exact projection. This scheme outperforms the fixed-rank schemes by orders of magnitude.}
     \label{fig:kpp}
\end{figure}

Figure \ref{fig:kpp} compares the MC run and a rank-15 approximation using our first-order optimal perturbative retraction (\ref{eq:corr1}) applied once or twice iteratively each time step. We also compute the solution using our automatic gradient-descent retraction from Algorithm \ref{alg:autoGradDescentRet} with our adaptive retraction from Algorithm \ref{alg:rankAugment}. 
This scheme is initialized by a rank-29 truncated SVD of the MC initial conditions and has hyperparameters of $r_{\text{inc}} = 5$, $r_{\text{max}} = 30$, $\sigma^* = 10^{-12}$, $\theta^* = 0.05$, $N_{\text{max}} = 8$, and $\Delta^* = 10^{-16}$. 
In addition, we compare these retractions with the exact projection operator (i.e.\ the truncated SVD) applied after every time step, which has no projection-retraction error $\epr$ but does incur dynamical model closure error $\ed$ (see table \ref{tab:keyExpressions}). 
We also compute the best possible approximation given by the truncated SVD of the MC solution, which only incurs normal closure error $\en$ (see Table \ref{tab:keyExpressions}).  In essence, the exact projection operator is the best approximation we can hope to make with the fixed-rank DLRA because it intrinsically does not have error accumulation. 
We see that the 2-iteration retraction significantly reduces the error that accumulates over time, but cannot match the best approximation at rank-15. With our robust, adaptive scheme, however, we are able to outperform the best rank-15  approximation despite the inherent error accumulation. The solution rank reduces after the initial conditions, which is reflected in the solution of the adaptive-rank algorithm. However, the fixed-rank DLRAs already accumulated error from the initial conditions. We also see that the error at $t=12.5$ seems to be concentrated at large values of $x$. This may also be remedied in the future by using weighted norms to reduce error in certain locations.

%-------------------------------------
%------------
\section{Conclusion}
\label{sec:conc}

We have derived two new sets of retractions, one building off of the other. Optimal perturbative retractions are an improvement over the original perturbative retractions developed in \cite{charous_lermusiaux_SJSC2023a}.  While we employ a perturbation series in $\Uo$ as before, we use the optimal $Z$ matrix, which incorporates higher-order nonlinearities in $\LLo$. 
The optimal perturbative retractions avoid error associated with overshoot via their optimal choice of coefficients $Z$ given the modes $U$, which is made possible without pseudoinversion by re-orthonormalization. The coefficient update equation (\ref{eq:zeq}) is equivalent to that used in proper orthogonal decomposition \cite{pod1,pod2} and other reduced-basis methods such as dynamic mode decomposition \cite{dmd1,dmd2,dmd3,dmd4}, but here the modes upon which the dynamics are projected are dynamically updated, enabling more accurate model-order reduction without the need for a large offline/training stage. 
The optimal perturbative retractions are more accurate in approximating the best low-rank approximation than their original counterparts and come at a slightly reduced computational cost. The dominating computational cost is $\BigO((m+n)r(r+r_L))$, where $r_L$ is the rank of the discretized differential operator $\LLo$; this is the same complexity as the original perturbative retractions \cite{charous_lermusiaux_SJSC2023a}, the projector-splitting integrator \cite{projSplit}, and the randomized singular value decomposition \cite{randSVD}.

Gradient-descent retractions are easily generated by repeatedly applying optimal perturbative retractions on the residual between a higher-rank system and our current approximation. If the full-rank system is, in fact, on the low-rank manifold, our gradient-descent retractions converge superlinearly to the best low-rank approximation. Otherwise, we expect linear convergence, but typically the rate of convergence is very fast and increases with the number of iterations applied. Our analysis complements that done for the projector-splitting integrator in the context of fixed-point iteration \cite{iterconv}. The gradient-descent retractions with higher-order perturbative retractions are analogous to higher-order generalizations of Newton's method on the low-rank manifold. Using an adaptive perturbative retraction proved very useful for the gradient-descent retractions since, when far away from the best low-rank approximation, low-order retractions can be more accurate, but as our approximation improves over each iteration, higher-order retractions become more efficient.

We have also shown how to robustify retractions in two ways. First, one may use the Moore-Penrose pseudoinverse for the inversion of the correlation matrix in the high-order optimal perturbative retractions. Second, we can robustify the first-order perturbative retraction by only preserving the span of the modes $U$. We incorporate this robust retraction into our gradient-descent scheme, obtaining retractions that converge rapidly to the exact projection operator. Pseudoinversion preserves mode continuity, a useful feature for interpretability, but introduces an extra source of error by truncating small singular values. Our latter approach that only preserves the span does not accumulate additional error, but it sacrifices mode continuity. Either scheme may be converted to a rank-adaptive integrator that allows for highly accurate solutions while adapting to the system rank for improved computational efficiency. When we correctly estimate (or underestimate) the rank of the system, the adaptive retraction with gradient-descent converges rapidly to the optimal low-rank approximation.
Otherwise, convergence can be slow due to a non-optimal choice of subspace $U$. 

Two novel Dynamically Orthogonal Runge-Kutta families of integrators were derived, 
the stable, optimal and 
gradient-descent Dynamically Orthogonal Runge-Kutta schemes,
so-DORK and gd-DORK, respectively. 
%the \textbf{a}daptive-rank \textbf{d}ynamically \textbf{o}rthogonal \textbf{R}unge-\textbf{K}utta st\textbf{able} (so-DORK) schemes and the \textbf{d}ynamically \textbf{o}rthogonal \textbf{R}unge-\textbf{K}utta gradient-\textbf{descent} (gd-DORK) schemes.
The so-DORK schemes build off of the optimal perturbative retractions. By writing the differential operator $\LLo$ as a perturbation series, we update, within the time step, the subspace in which we integrate and in which the system dynamics are projected. Similarly, the gd-DORK schemes update the subspace as we integrate, but they build off of the gradient-descent retractions. 

Finally, we verified the efficacy of our retractions and integrators on three examples. We first compared our schemes with the projected Runge-Kutta \cite{rkRets} and projector-splitting \cite{projSplit} integrators on a linear matrix differential equation. Next, we solved an advection-diffusion PDE, again comparing our retractions with those in the literature. We also showed the utility of our rank-adaptive algorithm using the so-DORK integrator. Finally, we used our gradient-descent retractions to build a low-rank implicit-explicit integrator for a nonlinear, stochastic diffusion-reaction PDE. We compared applying one and two iterations of gradient-descent at each time step. The 2-iteration gradient-descent retraction significantly reduced the error accumulation over time. We also compared the fixed-iteration schemes with our automatic algorithm with adaptive rank, which outperforms the fixed-rank methods partially because it is able to stably overapproximate the rank of the solution. 

In the future, we believe using weighted norms when we derive our retractions may allow us to strategically reduce errors in regions we care the most about \cite{lermusiaux_et_al_oceans2002,lermusiaux_PhysD2007,feppon_lermusiaux_SIMAX2019}. 
Another interesting direction is to use an adaptive deterministic or stochastic dimension in which the dimensions of the matrix we approximate changes in time rather than just the rank. To preserve high-order convergence when our low-rank approximation is ill-conditioned, a new approach could involve using an adaptive time step that depends on the condition number of the system. In addition, deriving implicit DORK schemes may prove beneficial for stiff systems. Finally, research into accelerating convergence while increasing the rank of the solution when the rank of $\LLo$ is large is needed.

\section*{Acknowledgments}
We thank Manan Doshi for his discussion on section \ref{sec:rankAdaptiveDerivation} and the entire MSEAS group for their collaboration.

\bibliographystyle{siamplain}
%\bibliography{bib_short,mseas_short}
\bibliography{bib_short,mseas}
\end{document}

%% file: shared_arxiv.tex
\usepackage{lipsum}
\usepackage{amsfonts}
\usepackage{graphicx}
\usepackage{epstopdf}
\usepackage{algorithmic}

\ifpdf
  \DeclareGraphicsExtensions{.eps,.pdf,.png,.jpg}
\else
  \DeclareGraphicsExtensions{.eps}
\fi

\newsiamremark{remark}{Remark}
\newsiamremark{hypothesis}{Hypothesis}
\crefname{hypothesis}{Hypothesis}{Hypotheses}
\newsiamthm{claim}{Claim}

\headers{Stable rank-adaptive DORK schemes}{Aaron Charous and Pierre F.J. Lermusiaux}

\title{Stable rank-adaptive \\Dynamically Orthogonal Runge-Kutta schemes\thanks{We thank the Office of Naval Research for partial support under grants N00014-19-1-2693 (IN-BDA) and N00014-19-1-2664 (Task Force Ocean, DEEP-AI).}}

\author{Aaron Charous\thanks{Mechanical Engineering, Center for Computational Science and Engineering, Massachusetts Institute of Technology, Cambridge, MA 02139 
  (\email{acharous@mit.edu}, \email{pierrel@mit.edu})}
\and Pierre F.J.\ Lermusiaux\footnotemark[2]}

\usepackage{graphicx}
\usepackage{amsmath}
\usepackage{epstopdf} %
\usepackage{mathptmx}      %
\usepackage{indentfirst}
\usepackage{amsfonts}
\usepackage{mathrsfs}  
\usepackage{bbm}
\usepackage{dsfont}
\usepackage{cancel}
\usepackage{xcolor}
\DeclareMathAlphabet{\mathpzc}{OT1}{pzc}{m}{it}
\usepackage{mathtools}
\usepackage{footnote}
\usepackage[titletoc]{appendix}
\usepackage{float}
\usepackage{hyperref}
\usepackage{caption}
\usepackage{subcaption}
\usepackage{threeparttable}
\DeclareMathAlphabet{\mathcal}{OMS}{cmsy}{m}{n}

\usepackage{amsopn}
\DeclareMathOperator{\diag}{diag}
\usepackage{amsmath}
\usepackage{dirtytalk}
\usepackage[section]{placeins}
\usepackage{algorithm}
\usepackage{algorithmic}
\usepackage{euscript}

\usepackage[letterpaper,textwidth=5.125in,textheight=8.3in]{geometry}

\usepackage{setspace}
\AtBeginDocument{%
  \addtolength\abovedisplayskip{-0.3\baselineskip}%
  \addtolength\belowdisplayskip{-0.3\baselineskip}%
  \addtolength\abovedisplayshortskip{-0.3\baselineskip}%
  \addtolength\belowdisplayshortskip{-0.3\baselineskip}%
}

\newcommand{\Stmr}{\mathcal{V}_{m,r}}
\newcommand{\Stnr}{\mathcal{V}_{n,r}}
\newcommand{\TStmr}{\mathcal{T}_U\mathcal{V}_{m,r}}
\newcommand{\TStnr}{\mathcal{T}_V\mathcal{V}_{n,r}}

\newcommand{\DOU}{\mathcal{U}_{m,r}}

\newcommand{\T}{\mathcal{T}}
\newcommand{\TX}{\T_{X}\mkern-2mu\Mr}
\newcommand{\TXn}{\T_{X_i}\mkern-2mu\Mr}

\newcommand{\TA}{\mathcal{T}_{A}\Mr}

\newcommand{\Mr}{\mathscr{M}_r}
\newcommand{\Mri}{\mathscr{M}_{r_{i}}}
\newcommand{\Mrii}{\mathscr{M}_{r_{i+1}}}

\newcommand{\du}{\delta_U}
\newcommand{\dv}{\delta_V}

\newcommand{\LL}{\mathscr{L}}

\newcommand{\epr}{\varepsilon_{pr}}
\newcommand{\ed}{\varepsilon_{D}}
\newcommand{\edf}{\varepsilon_{DF}}
\newcommand{\en}{\varepsilon_{\mathcal{N}}}
\newcommand{\el}{\varepsilon_{l}}
\newcommand{\et}{\varepsilon_{\text{tot}}}
\newcommand{\LLo}{\overline{\mathscr{L}}}
\newcommand{\Ld}{\overline{\EuScript{L}}}
\newcommand{\Reo}{G}
\DeclareRobustCommand{\rchi}{{\mathpalette\irchi\relax}}
\newcommand{\irchi}[2]{\raisebox{.85\depth}{$#1\chi$}} %
\newcommand{\newPoint}{\rchi}

\def\odU{\mathpalette\odUaux}
\def\odUaux#1#2{%
  \mkern2mu
  \setbox0=\hbox{\mathsurround=0pt$#1\overline{\mkern-1mu #2}$}%
  \setbox1=\hbox to \wd0{\hss$#1\cdot$\hss}%
  \vbox{\offinterlineskip\copy1\vskip-.69\ht1\box0}%
}

\def\odZ{\mathpalette\odZaux}
\def\odZaux#1#2{%
  \mkern2mu
  \setbox0=\hbox{\mathsurround=0pt$#1\overline{\mkern-3mu #2}$}%
  \setbox1=\hbox to \wd0{\hss$#1\cdot$\hss}%
  \vbox{\offinterlineskip\copy1\vskip-.69\ht1\box0}%
}
\newcommand{\Uo}{\odU{U}}
\newcommand{\Zo}{\odZ{Z}}

\def\odUUaux#1#2{%
  \mkern2mu
  \setbox0=\hbox{\mathsurround=0pt$#1\tilde{\mkern-1mu #2}$}%
  \setbox1=\hbox to \wd0{\hss$#1\cdot$\hss}%
  \vbox{\offinterlineskip\copy1\vskip-.69\ht1\box0}%
}

\def\odZZaux#1#2{%
  \mkern2mu
  \setbox0=\hbox{\mathsurround=0pt$#1\tilde{\mkern-3mu #2}$}%
  \setbox1=\hbox to \wd0{\hss$#1\cdot$\hss}%
  \vbox{\offinterlineskip\copy1\vskip-.665\ht1\box0}%
}

\newcommand{\PU}{\mathscr{P}_U}
\newcommand{\PV}{\mathscr{P}_V}
\newcommand{\PUp}{\mathscr{P}_U^\perp}
\newcommand{\PUpi}{\mathscr{P}_{U_i}^\perp}
\newcommand{\PVp}{\mathscr{P}_V^\perp}
\newcommand{\PUnp}{\mathscr{P}_{U_i}^\perp}

\newcommand{\PTA}{\mathscr{P}_{\TA}}

\newcommand{\PT}{\mathscr{P}_{\TX}}
\newcommand{\PTn}{\mathscr{P}_{\TXn}}

\newcommand{\PM}{\mathscr{P}_{\mkern-10mu\Mr}}

\newcommand{\PMrii}{\mathscr{P}_{\mkern-10mu\Mrii}}
\newcommand{\PMp}{\mathscr{P}^\perp_{\mkern-10mu\Mr}}

\newcommand{\BigO}{\mathcal{O}}
\newcommand{\Ret}{\mathcal{R}}
\newcommand{\Rr}{\mathds{R}_*^{n\times r}}
\newcommand{\Rrr}{\mathds{R}_*^{r\times r}}

\newcommand{\Ud}{\skew{5}{\Dot}{\mathpzc{U}}}
\newcommand{\Zd}{\Dot{\mathcal{Z}}}

\newcommand{\Xf}{\mathfrak{X}} %
\newcommand{\Xeither}{\mathscr{X}}

\DeclareMathOperator*{\argmin}{arg\,min}

\makeatletter
\newcommand*{\addFileDependency}[1]{%
  \typeout{(#1)}%
  \@addtofilelist{#1}%
  \IfFileExists{#1}{}{\typeout{No file #1.}}%
}
\makeatother

\newcommand*{\myexternaldocument}[1]{%
    \externaldocument{#1}%
    \addFileDependency{#1.tex}%
    \addFileDependency{#1.aux}%
}

%% file: siam_arxiv.bbl
\begin{thebibliography}{10}

\bibitem{retSurvey}
{\sc P.-A. Absil and I.~V. Oseledets}, {\em Low-rank retractions: a survey and
  new results}, Computational Optimization and Applications, 62 (2015),
  pp.~5--29.

\bibitem{besiktepe_et_al_JMS2003}
{\sc {\c S}.~T. {Be{\c s}iktepe}, P.~F.~J. {Lermusiaux}, and A.~R. {Robinson}},
  {\em {Coupled physical and biogeochemical data-driven simulations of
  Massachusetts Bay in late summer: Real-time and post-cruise data
  assimilation}}, Journal of Marine Systems, 40--41 (2003), pp.~171--212,
  \url{https://doi.org/10.1016/S0924-7963(03)00018-6}.

\bibitem{pod2}
{\sc G.~Berkooz, P.~Holmes, and J.~L. Lumley}, {\em The proper orthogonal
  decomposition in the analysis of turbulent flows}, Annual review of fluid
  mechanics, 25 (1993), pp.~539--575.

\bibitem{localCoord1}
{\sc M.~Billaud-Friess, A.~Falc{\'o}, and A.~Nouy}, {\em A geometry based
  algorithm for dynamical low-rank approximation}, arXiv preprint
  arXiv:2001.08599,  (2020).

\bibitem{localCoord2}
{\sc M.~Billaud-Friess, A.~Falc{\'o}, and A.~Nouy}, {\em A new splitting
  algorithm for dynamical low-rank approximation motivated by the fibre bundle
  structure of matrix manifolds}, BIT Numerical Mathematics,  (2021),
  pp.~1--22.

\bibitem{celledoni2002class}
{\sc E.~Celledoni and B.~Owren}, {\em A class of intrinsic schemes for
  orthogonal integration}, SIAM Journal on Numerical Analysis, 40 (2002),
  pp.~2069--2084.

\bibitem{ceruti2022rank}
{\sc G.~Ceruti, J.~Kusch, and C.~Lubich}, {\em A rank-adaptive robust
  integrator for dynamical low-rank approximation}, BIT Numerical Mathematics,
  (2022), pp.~1--26.

\bibitem{kslVariant}
{\sc G.~Ceruti and C.~Lubich}, {\em An unconventional robust integrator for
  dynamical low-rank approximation}, BIT Numerical Mathematics,  (2021),
  pp.~1--22.

\bibitem{charous_MSThesis2021}
{\sc A.~Charous}, {\em High-order retractions for reduced-order modeling and
  uncertainty quantification}, master's thesis, Massachusetts Institute of
  Technology, Center for Computational Science and Engineering, Cambridge,
  Massachusetts, Feb. 2021.

\bibitem{charous_PhDThesis2023}
{\sc A.~Charous}, {\em Dynamical Reduced-Order Models for High-Dimensional
  Systems}, PhD thesis, Massachusetts Institute of Technology, Department of
  Mechanical Engineering and Center for Computational Science and Engineering,
  Cambridge, Massachusetts, June 2023.

\bibitem{charous_lermusiaux_Oceans2021}
{\sc A.~Charous and P.~F.~J. Lermusiaux}, {\em Dynamically orthogonal
  differential equations for stochastic and deterministic reduced-order
  modeling of ocean acoustic wave propagation}, in OCEANS 2021 IEEE/MTS, IEEE,
  Sept. 2021, pp.~1--7,
  \url{https://doi.org/10.23919/OCEANS44145.2021.9705914}.

\bibitem{charous_lermusiaux_SJSC2023a}
{\sc A.~Charous and P.~F.~J. Lermusiaux}, {\em {D}ynamically {O}rthogonal
  {R}unge--{K}utta schemes with perturbative retractions for the dynamical
  low-rank approximation}, {SIAM} Journal on Scientific Computing, 45 (2023),
  pp.~A872--A897, \url{https://doi.org/10.1137/21M1431229}.

\bibitem{charous_lermusiaux_SJSC2023c}
{\sc A.~Charous and P.~F.~J. Lermusiaux}, {\em Implicit {D}ynamically
  {O}rthogonal {R}unge--{K}utta schemes and efficient nonlinearity evaluation},
  {SIAM} Journal on Scientific Computing,  (2023).
\newblock To be submitted.

\bibitem{dirac}
{\sc P.~A.~M. Dirac}, {\em Note on exchange phenomena in the thomas atom},
  Mathematical Proceedings of the Cambridge Philosophical Society, 26 (1930),
  p.~376–385.

\bibitem{fkppeco}
{\sc M.~El-Hachem, S.~W. McCue, W.~Jin, Y.~Du, and M.~J. Simpson}, {\em
  Revisiting the fisher--kolmogorov--petrovsky--piskunov equation to interpret
  the spreading--extinction dichotomy}, Proceedings of the Royal Society A, 475
  (2019), p.~20190378.

\bibitem{feppon_lermusiaux_SIREV2018}
{\sc F.~Feppon and P.~F.~J. Lermusiaux}, {\em Dynamically orthogonal numerical
  schemes for efficient stochastic advection and {L}agrangian transport},
  {SIAM} Review, 60 (2018), pp.~595--625,
  \url{https://doi.org/10.1137/16M1109394}.

\bibitem{feppon_lermusiaux_SIMAX2018a}
{\sc F.~Feppon and P.~F.~J. Lermusiaux}, {\em A geometric approach to dynamical
  model-order reduction}, SIAM Journal on Matrix Analysis and Applications, 39
  (2018), pp.~510--538, \url{https://doi.org/10.1137/16M1095202}.

\bibitem{feppon_lermusiaux_SIMAX2019}
{\sc F.~Feppon and P.~F.~J. Lermusiaux}, {\em The extrinsic geometry of
  dynamical systems tracking nonlinear matrix projections}, {SIAM} Journal on
  Matrix Analysis and Applications, 40 (2019), pp.~814--844,
  \url{https://doi.org/10.1137/18M1192780}.

\bibitem{fisher}
{\sc R.~A. Fisher}, {\em The wave of advance of advantageous genes}, Annals of
  eugenics, 7 (1937), pp.~355--369.

\bibitem{gao2022riemannian}
{\sc B.~Gao and P.-A. Absil}, {\em A riemannian rank-adaptive method for
  low-rank matrix completion}, Computational Optimization and Applications, 81
  (2022), pp.~67--90.

\bibitem{gottlieb1998total}
{\sc S.~Gottlieb and C.-W. Shu}, {\em Total variation diminishing runge-kutta
  schemes}, Mathematics of computation, 67 (1998), pp.~73--85.

\bibitem{reacdiff}
{\sc P.~Grindrod}, {\em The theory and applications of reaction-diffusion
  equations: patterns and waves}, Clarendon Press, 1996.

\bibitem{gupta_et_al_Oceans2019}
{\sc A.~Gupta, P.~J. Haley, D.~N. Subramani, and P.~F.~J. Lermusiaux}, {\em
  Fish modeling and {B}ayesian learning for the {L}akshadweep {I}slands}, in
  OCEANS 2019 MTS/IEEE SEATTLE, Seattle, Oct. 2019, IEEE, pp.~1--10,
  \url{https://doi.org/10.23919/OCEANS40490.2019.8962892}.

\bibitem{gupta_lermusiaux_PRSA2021}
{\sc A.~Gupta and P.~F.~J. Lermusiaux}, {\em Neural closure models for
  dynamical systems}, Proceedings of The Royal Society A, 477 (2021),
  pp.~1--29, \url{https://doi.org/10.1098/rspa.2020.1004}.

\bibitem{projectionMethods1}
{\sc E.~Hairer}, {\em Symmetric projection methods for differential equations
  on manifolds}, BIT Numerical Mathematics, 40 (2000), pp.~726--734.

\bibitem{projectionMethods2}
{\sc E.~Hairer}, {\em Solving differential equations on manifolds}, Lec.\
  Notes, Univ.\ de Geneve,  (2011).

\bibitem{randSVD}
{\sc N.~Halko, P.-G. Martinsson, and J.~A. Tropp}, {\em Finding structure with
  randomness: Probabilistic algorithms for constructing approximate matrix
  decompositions}, SIAM review, 53 (2011), pp.~217--288.

\bibitem{hauck2022predictor}
{\sc C.~Hauck and S.~Schnake}, {\em A predictor-corrector strategy for
  adaptivity in dynamical low-rank approximations}, arXiv preprint
  arXiv:2209.00550,  (2022).

\bibitem{hesthaven2022rank}
{\sc J.~S. Hesthaven, C.~Pagliantini, and N.~Ripamonti}, {\em Rank-adaptive
  structure-preserving model order reduction of hamiltonian systems}, ESAIM:
  Mathematical Modelling and Numerical Analysis, 56 (2022), pp.~617--650.

\bibitem{hochbruckdynamical}
{\sc M.~Hochbruck, M.~Neher, and S.~Schrammer}, {\em Dynamical low-rank
  integrators for second-order matrix differential equations}, BIT Numerical
  Mathematics, 63 (2023).

\bibitem{imex}
{\sc W.~Hundsdorfer and J.~G. Verwer}, {\em Numerical solution of
  time-dependent advection-diffusion-reaction equations}, vol.~33, Springer
  Science \& Business Media, 2007.

\bibitem{rkRets}
{\sc E.~Kieri and B.~Vandereycken}, {\em Projection methods for dynamical
  low-rank approximation of high-dimensional problems}, Computational Methods
  in Applied Mathematics, 19 (2019), pp.~73--92.

\bibitem{dlra}
{\sc O.~Koch and C.~Lubich}, {\em Dynamical low-rank approximation}, SIAM J. on
  Matrix Analysis and Applications, 29 (2007), pp.~434--454.

\bibitem{iterconv}
{\sc D.~Kolesnikov and I.~V. Oseledets}, {\em Convergence analysis of projected
  fixed-point iteration on a low-rank matrix manifold}, Numerical Linear
  Algebra with Applications, 25 (2018), p.~e2140.

\bibitem{kpp}
{\sc A.~Kolmogorov, I.~Petrovsky, and N.~Piscounov}, {\em A study of the
  diffusion equation with increase in the amount of substance}, Bull. Univ.
  {\'E}tat Moscou, 1 (1937).

\bibitem{dmd4}
{\sc J.~N. Kutz, S.~L. Brunton, B.~W. Brunton, and J.~L. Proctor}, {\em Dynamic
  mode decomposition: data-driven modeling of complex systems}, SIAM, 2016.

\bibitem{lermusiaux_PhDThesis1997}
{\sc P.~F.~J. Lermusiaux}, {\em Error subspace data assimilation methods for
  ocean field estimation: {T}heory, validation and applications}, PhD thesis,
  Harvard University, Cambridge, MA, May 1997.

\bibitem{lermusiaux_MWR1999}
{\sc P.~F.~J. Lermusiaux}, {\em Data assimilation via {E}rror {S}ubspace
  {S}tatistical {E}stimation, part {II}: Mid-{A}tlantic {B}ight shelfbreak
  front simulations, and {ESSE} validation}, Monthly Weather Review, 127
  (1999), pp.~1408--1432,
  \url{https://doi.org/10.1175/1520-0493(1999)127<1408:DAVESS>2.0.CO;2}.

\bibitem{lermusiaux_JMS2001}
{\sc P.~F.~J. Lermusiaux}, {\em Evolving the subspace of the three-dimensional
  multiscale ocean variability: {M}assachusetts {B}ay}, Journal of Marine
  Systems, 29 (2001), pp.~385--422,
  \url{https://doi.org/10.1016/S0924-7963(01)00025-2}.

\bibitem{lermusiaux_PhysD2007}
{\sc P.~F.~J. {Lermusiaux}}, {\em Adaptive modeling, adaptive data assimilation
  and adaptive sampling}, Physica D: Nonlinear Phenomena, 230 (2007),
  pp.~172--196, \url{https://doi.org/10.1016/j.physd.2007.02.014}.

\bibitem{lermusiaux_2.29_notes}
{\sc P.~F.~J. Lermusiaux}, {\em Numerical fluid mechanics}.
\newblock {MIT} OpenCourseWare, May 2015,
  \url{https://ocw.mit.edu/courses/mechanical-engineering/2-29-numerical-fluid-mechanics-spring-2015/lecture-notes-and-references/}.

\bibitem{lermusiaux_et_al_Oceanog2011}
{\sc P.~F.~J. Lermusiaux, P.~J. Haley, W.~G. Leslie, A.~Agarwal, O.~Logutov,
  and L.~J. Burton}, {\em Multiscale physical and biological dynamics in the
  {P}hilippine {A}rchipelago: Predictions and processes}, Oceanography, 24
  (2011), pp.~70--89, \url{https://doi.org/10.5670/oceanog.2011.05}.
\newblock {S}pecial Issue on the Philippine Straits Dynamics Experiment.

\bibitem{lermusiaux_robinson_MWR1999}
{\sc P.~F.~J. Lermusiaux and A.~R. Robinson}, {\em Data assimilation via
  {E}rror {S}ubspace {S}tatistical {E}stimation, part {I}: Theory and schemes},
  Monthly Weather Review, 127 (1999), pp.~1385--1407,
  \url{https://doi.org/10.1175/1520-0493(1999)127<1385:DAVESS>2.0.CO;2}.

\bibitem{lermusiaux_et_al_oceans2002}
{\sc P.~F.~J. Lermusiaux, A.~R. Robinson, P.~J. Haley, and W.~G. Leslie}, {\em
  Advanced interdisciplinary data assimilation: Filtering and smoothing via
  error subspace statistical estimation}, in Proceedings of The OCEANS 2002
  MTS/IEEE conference, Holland Publications, 2002, pp.~795--802,
  \url{https://doi.org/10.1109/oceans.2002.1192071}.

\bibitem{eco}
{\sc S.~A. Levin, .~H.~C. Muller-Landau, .~R. Nathan, and .~J. Chave}, {\em The
  ecology and evolution of seed dispersal: a theoretical perspective}, Ann.
  Rev. of Eco., Evol., and Sys., 34 (2003), pp.~575--604.

\bibitem{lin_PhDThesis2020}
{\sc J.~Lin}, {\em Bayesian Learning for High-Dimensional Nonlinear Systems:
  {M}ethodologies, Numerics and Applications to Fluid Flows}, PhD thesis,
  Massachusetts Institute of Technology, Department of Mechanical Engineering,
  Cambridge, Massachusetts, Sept. 2020.

\bibitem{lin_lermusiaux_NM2021}
{\sc J.~Lin and P.~F.~J. Lermusiaux}, {\em Minimum-correction second-moment
  matching: {T}heory, algorithms and applications}, Numerische Mathematik, 147
  (2021), pp.~611--650, \url{https://doi.org/10.1007/s00211-021-01178-8}.

\bibitem{lubichbook}
{\sc C.~Lubich}, {\em From quantum to classical molecular dynamics: reduced
  models and numerical analysis}, European Mathematical Society, 2008.

\bibitem{projSplit}
{\sc C.~Lubich and I.~V. Oseledets}, {\em A projector-splitting integrator for
  dynamical low-rank approximation}, BIT Numerical Mathematics, 54 (2014),
  pp.~171--188.

\bibitem{pod1}
{\sc Z.~Luo and G.~Chen}, {\em Proper orthogonal decomposition methods for
  partial differential equations}, Academic Press, 2018.

\bibitem{musharbash2020symplectic}
{\sc E.~Musharbash, F.~Nobile, and E.~Vidli{\v{c}}kov{\'a}}, {\em Symplectic
  dynamical low rank approximation of wave equations with random parameters},
  BIT Numerical Mathematics, 60 (2020), pp.~1153--1201.

\bibitem{pagliantini2021dynamical}
{\sc C.~Pagliantini}, {\em Dynamical reduced basis methods for hamiltonian
  systems}, Numerische Mathematik, 148 (2021), pp.~409--448.

\bibitem{dmd1}
{\sc C.~W. Rowley, I.~Mezi{\'c}, S.~Bagheri, P.~Schlatter, and D.~S.
  Henningson}, {\em Spectral analysis of nonlinear flows}, Journal of fluid
  mechanics, 641 (2009), pp.~115--127.

\bibitem{ryu_et_al_Oceans2021}
{\sc T.~Ryu, J.~P. Heuss, P.~J. Haley, Jr., C.~Mirabito, E.~Coelho, P.~Hursky,
  M.~C. Sch\"onau, K.~Heaney, and P.~F.~J. Lermusiaux}, {\em Adaptive
  stochastic reduced order modeling for autonomous ocean platforms}, in OCEANS
  2021 IEEE/MTS, IEEE, Sept. 2021, pp.~1--9,
  \url{https://doi.org/10.23919/OCEANS44145.2021.9705790}.

\bibitem{sapsis_lermusiaux_PHYSD2012}
{\sc T.~P. Sapsis and P.~F.~J. Lermusiaux}, {\em Dynamical criteria for the
  evolution of the stochastic dimensionality in flows with uncertainty},
  Physica D: Nonlinear Phenomena, 241 (2012), pp.~60--76,
  \url{https://doi.org/10.1016/j.physd.2011.10.001}.

\bibitem{dmd2}
{\sc P.~J. Schmid}, {\em Dynamic mode decomposition of numerical and
  experimental data}, Journal of fluid mechanics, 656 (2010), pp.~5--28.

\bibitem{bio}
{\sc J.~A. Sherratt and J.~D. Murray}, {\em Models of epidermal wound healing},
  Proceedings of the Royal Society of London. Series B: Biological Sciences,
  241 (1990), pp.~29--36.

\bibitem{ALS}
{\sc A.~Szlam, A.~Tulloch, and M.~Tygert}, {\em Accurate low-rank
  approximations via a few iterations of alternating least squares}, SIAM J. on
  Matrix Analysis and Applications, 38 (2017), pp.~425--433.

\bibitem{trefethen1997numerical}
{\sc L.~Trefethen and D.~Bau}, {\em Numerical Linear Algebra}, SIAM, 1997.

\bibitem{dmd3}
{\sc J.~H. Tu}, {\em Dynamic mode decomposition: Theory and applications}, PhD
  thesis, Princeton University, 2013.

\bibitem{ueckermann_et_al_JCP2013}
{\sc M.~P. Ueckermann, P.~F.~J. Lermusiaux, and T.~P. Sapsis}, {\em Numerical
  schemes for dynamically orthogonal equations of stochastic fluid and ocean
  flows}, Journal of Computational Physics, 233 (2013), pp.~272--294,
  \url{https://doi.org/10.1016/j.jcp.2012.08.041}.

\bibitem{weyl}
{\sc H.~Weyl}, {\em Das asymptotische verteilungsgesetz der eigenwerte linearer
  partieller differentialgleichungen (mit einer anwendung auf die theorie der
  hohlraumstrahlung)}, Math. Ann., 71 (1912), pp.~441--479.

\bibitem{zhou2016riemannian}
{\sc G.~Zhou, W.~Huang, K.~A. Gallivan, P.~Van~Dooren, and P.-A. Absil}, {\em A
  riemannian rank-adaptive method for low-rank optimization}, Neurocomputing,
  192 (2016), pp.~72--80.

\bibitem{zimmermann2018geometric}
{\sc R.~Zimmermann, B.~Peherstorfer, and K.~Willcox}, {\em Geometric subspace
  updates with applications to online adaptive nonlinear model reduction}, SIAM
  Journal on Matrix Analysis and Applications, 39 (2018), pp.~234--261.

\end{thebibliography}
